\documentclass[a4paper, 11pt]{amsart}
\usepackage{graphicx}
\usepackage{amssymb}
\usepackage{epstopdf}
\usepackage{fancyhdr}
\usepackage{amsmath}
\usepackage{amsthm}
\usepackage{color}

\usepackage[top=2.2 cm, bottom=2.2 cm,left=2. cm, right=2. cm]{geometry}
\def\AMS{{\it AMS subject classifications: }}

\definecolor{purple}{rgb}{.9,0,.9}

\newcommand{\Real}{\mathbb{R}}

\newcommand{\rfa}{\qquad {\rm for \ all}\ }

\newcommand{\bphi}{\boldsymbol{\phi}}

\newcommand{\cA}{{\mathcal A}}
\newcommand{\cE}{{\mathcal E}}
\newcommand{\cI}{{\mathcal I}}

\newcommand{\cN}{{\mathcal N}}

\newcommand{\cT}{{\mathcal T}}\newcommand{\cU}{{\mathcal U}}
\newcommand{\cV}{{\mathcal V}}\newcommand{\cW}{{\mathcal W}}

\newcommand{\bb}{{\bf b}}
\newcommand{\be}{{\bf e}}\newcommand{\bff}{{\bf f}}

\newcommand{\bu}{{\bf u}}

\newcommand{\bA}{{\bf A}}

\newcommand{\p}{\textbf{p}}
\newcommand{\n}{\textbf{n}}
\newcommand{\F}{\textbf{F}}
\newcommand{\R}{\textbf{R}}
\newcommand{\G}{\textbf{G}}

\newcommand{\J}{\textbf{J}}

\newtheorem{theorem}{Theorem}[section]
\newtheorem{lemma}[theorem]{Lemma}

\newtheorem{definition}[theorem]{Definition}

\newtheorem{assumption}{Assumption}

\newcommand{\beqn}{\begin{equation}}
\newcommand{\eeqn}{\end{equation}}

\newcommand{\norm}[1]{\left\|#1\right \|}

\newcommand{\bbE}{\mathbb{E}}
\newcommand{\bbV}{\mathbb{V}}
\newcommand{\bnu}{\boldsymbol{\nu}}
\newcommand{\ve}{\varepsilon}

\def\Xint#1{\mathchoice
{\XXint\displaystyle\textstyle{#1}}%
{\XXint\textstyle\scriptstyle{#1}}%
{\XXint\scriptstyle\scriptscriptstyle{#1}}%
{\XXint\scriptscriptstyle\scriptscriptstyle{#1}}%
\!\int}
\def\XXint#1#2#3{{\setbox0=\hbox{$#1{#2#3}{\int}$ }
\vcenter{\hbox{$#2#3$ }}\kern-.58\wd0}}

\def\dashint{\Xint-}

\title{Homogenization of a viscoelastic model for  plant cell wall biomechanics$^\ast$}
\author{Mariya Ptashnyk$^1$ and Brian Seguin$^2$}

\thanks{$^1$Department  of Mathematics, University of Dundee, Dundee DD1 4HN, UK, m.ptashnyk@dundee.ac.uk \\
 $^2$Department of Mathematics and Statistics, Loyola University Chicago, Chicago 60660 USA, bseguin@luc.edu\\
$^\ast$M.\ Ptashnyk and B.\ Seguin gratefully acknowledge the support of the EPSRC First Grant EP/K036521/1 ``Multiscale modelling and analysis of mechanical properties of plant cells and tissues''.}

\begin{document}

\maketitle

\begin{abstract} 
The microscopic structure of a plant cell wall is given by cellulose microfibrils embedded in a cell wall matrix. 
In this paper we  consider a  microscopic model  for interactions between viscoelastic deformations of a plant cell wall and chemical processes in  the cell wall matrix. We consider elastic  deformations of the cell wall microfibrils and viscoelastic Kelvin--Voigt type deformations of the cell wall matrix. Using homogenization  techniques  (two-scale convergence and periodic unfolding methods) we derive   macroscopic equations from the microscopic model for  cell wall biomechanics consisting of strongly coupled equations of linear  viscoelasticity and a system of reaction-diffusion and ordinary differential equations. As  is typical for  microscopic viscoelastic problems,  the macroscopic equations for viscoelastic deformations of plant cell walls  contain memory terms. The derivation of the macroscopic problem for degenerate viscoelastic equations is conducted using a perturbation argument.
 \end{abstract}

\vspace{0.5 cm}

 \AMS{\small 35B27, 35Q92, 35Kxx, 74Qxx, 74A40,  74D05}

{\small  {\it Key words:} Homogenization; two-scale convergence; periodic unfolding method; viscoelasticity; 

plant modelling.
%


\section*{Introduction}

To obtain a better understanding of the mechanical properties and development of plant tissues it is important to model and analyse the interactions between the chemical processes and mechanical deformations of plant cells. The main feature of plant cells are their walls, which must be strong to resist high internal hydrostatic pressure (turgor pressure) and flexible to permit growth.  The biomechanics of plant cell walls is determined by the cell wall microstructure,  given by   microfibrils, and  the physical  properties of the cell
wall matrix.    The orientation of microfibrils, their length, high tensile strength, and interaction with wall matrix macromolecules strongly influences the wall's stiffness.   It is also supposed that calcium-pectin cross-linking chemistry is one of the main regulators of cell wall elasticity and extension \cite{WHH}.  Pectin can be modified  by the enzyme pectin methylesterase (PME), which removes methyl groups by breaking ester bonds. The de-esterified pectin  is able to form  calcium-pectin cross-links, and so stiffen the cell wall and reduce its expansion, see e.g.~\cite{WG}.  It has been shown that   the    modification of pectin  by PME and  the control of the amount of  calcium-pectin cross-links  greatly influence the  mechanical deformations of plant cell walls \cite{P2011,P2010},  and the interference with PME activity causes dramatic changes in growth behavior of plant cells and tissues \cite{Wolf}.
%

To address the interactions between chemistry and mechanics, in the microscopic model for plant cell wall biomechanics we consider  the influence of the microstructure, associated with the cellulose microfibrils,  and  the calcium-pectin cross-links on the mechanical properties of  plant  cell walls.  
We model the cell wall as a three-dimensional continuum consisting of a polysaccharide matrix embedded with cellulose microfibrils.  Within the matrix, we consider the dynamics of the enzyme PME, methylesterfied pectin, demethylesterfied pectin, calcium ions, and calcium-pectin cross-links.  It was observed experimentally that plant cell wall microfibrils are anisotropic, see e.g.\ \cite{DMKM}, and   the cell wall matrix  in addition  to elastic deformations  exhibits viscous behaviour, see e.g.\ \cite{Hayot}. Hence we model  the cell wall matrix  as a linearly  viscoelastic Kelvin--Voigt material, whereas microfibrils are modelled as an anisotropic  linearly elastic material.  The model for plant cell wall biomechanics in which   the cell wall matrix was assumed to be a linearly elastic was derived and analysed in \cite{PS}. The interplay between the mechanics and the cross-link dynamics comes in by assuming that the elastic and viscous properties of the cell wall matrix depend on the density of the cross-links and that stress within the cell wall can break calcium-pectin cross-links.  
 The  stress-dependent opening of calcium channels in the cell plasma membrane is addressed in the flux boundary conditions for calcium ions.
The resulting microscopic model is a system of strongly coupled four diffusion-reaction equations, one ordinary differential equation,  and the equations of  linear viscoelasticity. Since only the cell wall matrix is viscoelastic we obtain degenerate elastic-viscoelastic equations. In our model we focus on the interactions between the chemical reactions within the cell wall and its deformation and, hence, do not consider the growth of the cell wall.  

 To analyse the macroscopic mechanical properties of the plant cell wall we rigorously derive  macroscopic equations from the microscopic description of plant cell wall biomechanics. 
   The two-scale convergence, e.g.~\cite{allaire,Nguetseng}, and the periodic unfolding  method, e.g.~\cite{CDG,CDDGZ}, are applied to obtain the macroscopic equations. For the  viscoelastic equations the macroscopic momentum balance equation contains a term that depends on the history of the strain  represented by an integral term (fading memory effect).    Due to the coupling between the viscoelastic properties and the biochemistry of a plant cell wall,   the  elastic and viscous tensors depend on space and time variables. This fact introduces additional complexity in the derivation and  in the structure of the macroscopic   equations, compered to classical viscoelastic equations.

The main novelty of this paper is  the  multiscale analysis and derivation of the macroscopic problem   from a microscopic description of the mechanical and chemical processes. 
This approach allows us to take into account the complex microscopic structure of a plant cell wall and  to analyze the impact of the heterogeneous distribution of cell wall structural elements on the mechanical properties and development of plants.   
The main mathematical difficulty arises from  the   strong coupling between the equations of linear  viscoelasticity  for cell wall mechanics  and the system of reaction-diffusion and ordinary differential  equations for the chemical processes in the wall matrix. Also   the degeneracy of the  viscoelastic equations, due to the fact that   only the cell wall matrix is assumed to be viscoelastic  and  microfibres are assumed to be elastic, induces additional technical diffuculties.

  A multiscale analysis of the viscoelastic  equations with time-independent coefficients  was  considered previously  in \cite{FS,GPX,Mas1987,SP}.    Macroscopic equations for scalar elastic-viscoelastic equations with time-independent coefficients were derived in \cite{Ene} by applying the H-convergence method \cite{Murat}.    A microscopic viscoelastic Kelvin--Voigt model with time-dependent coefficients  in the context of thermo-viscoelasticity was analyzed in \cite{AKPS}.  Macroscopic equations were derived by applying the method of asymptotic expansion. 

The paper is organised as follows. In Section~\ref{sectmodel} we formulate a mathematical model   for plant cell wall biomechanics in which  the cell wall matrix is assumed to be  viscoelastic. In Section~\ref{mainresults} we summarise the main results of the paper. The well-possednes of the microscopic model is shown in Section~\ref{existence}.  The multiscale analysis   of the microscopic model is conducted in Section~\ref{secthom}. 
Since we assume that only the cell wall matrix exhibits viscoelastic  behaviour and microfibrils are elastic, the viscous tensor is zero in the domain occupied by the microfibrils. This  fact causes  technical difficulties in the multiscale analysis of the microscopic model.  To derive the macroscopic equations for the elastic-viscoelastic model for cell wall biomechanics we first consider  perturbed equations by introducing an inertial term. Then, letting the perturbation parameter in the macroscopic model tend to zero,  we obtain the effective homogenized equations for the original  elastic-viscoelastic model. 

\section{Microscopic model for viscoelastic deformations of plant cell walls}\label{sectmodel}

 It was observed experimentally   that in addition to elastic deformations   the plant cell wall matrix  exhibit viscoelastic behaviour \cite{Hayot}. Hence,  in contrast to the problem considered in \cite{PS},  here  we assume that the deformation in the plant cell wall matrix  is determined by the equations of linear  viscoelasticity. 

We consider a  domain $\Omega=(0,a_1)\times(0,a_2)\times(0,a_3)$ representing a part of a plant cell wall, where $a_i$, $i=1,2,3$, are positive numbers and  the microfibrils are oriented in the $x_3$-direction, see Fig.~\ref{FigDomain1}(a).
The part of $\partial\Omega$ on the exterior of the cell wall   is given by
$
\Gamma_{\cE}= \{a_1\}\times(0, a_2)\times(0,a_3),
$
and the interior boundary $\Gamma_\cI$ of the cell wall is given by
$
\Gamma_{\cI}=\{0\}\times(0, a_2)\times (0,a_3).
$
The top and bottom boundaries are defined by
$\Gamma_{\cU} = (0, a_1) \times\{ 0 \} \times (0, a_3) \cup (0, a_1) \times \{ a_2 \} \times (0, a_3)$. 

To determine the  microscopic structure of the cell wall, we consider $Y= (0,1)^2\times (0,a_3)$ and  define  $\hat Y = (0,1)^2$, and a subdomain $\hat Y_F$ with  $\overline {\hat Y_F}\subset \hat Y$ and $\hat Y_M = \hat Y \setminus \overline{\hat Y_F}$,  see Fig.~\ref{FigDomain1}(b).  Then $Y_F= \hat Y_F \times (0, a_3)$ and $Y_M= \hat Y_M \times (0, a_3)$ represent the cell wall microfibrils and cell wall matrix. We also define $\hat \Gamma = \partial \hat Y_F$ and $\Gamma = \partial Y_F$. 

We assume that the  microfibrils in the cell wall  are distributed periodically  and have a diameter on the order of $\ve$, where  the small parameter $\ve$ characterise the size of the microstructure.
The domains
\begin{equation*}
\Omega_F^\ve =\bigcup_{\xi \in  \mathbb Z^2} \big\{ \ve (\hat Y_F+ \xi) \times(0, a_3)\;\; | \;\;   \ve (\hat Y+ \xi)\subset  ( 0,a_1) \times (0 , a_2 ) \big\}
\quad \text{ and } \quad  \Omega_M^\ve =\Omega \, \setminus \overline {\Omega_F^\ve}
\end{equation*}
denote the part of $\Omega$ occupied by the microfibrils and  by the cell wall matrix, respectively.  The boundary between the matrix and the microfibrils is denoted by
\begin{equation*}
\Gamma^\ve =\partial\Omega_M^\ve\cap\partial\Omega_F^\ve.
\end{equation*} 

\begin{figure}[t]
\includegraphics[width=4.5in]{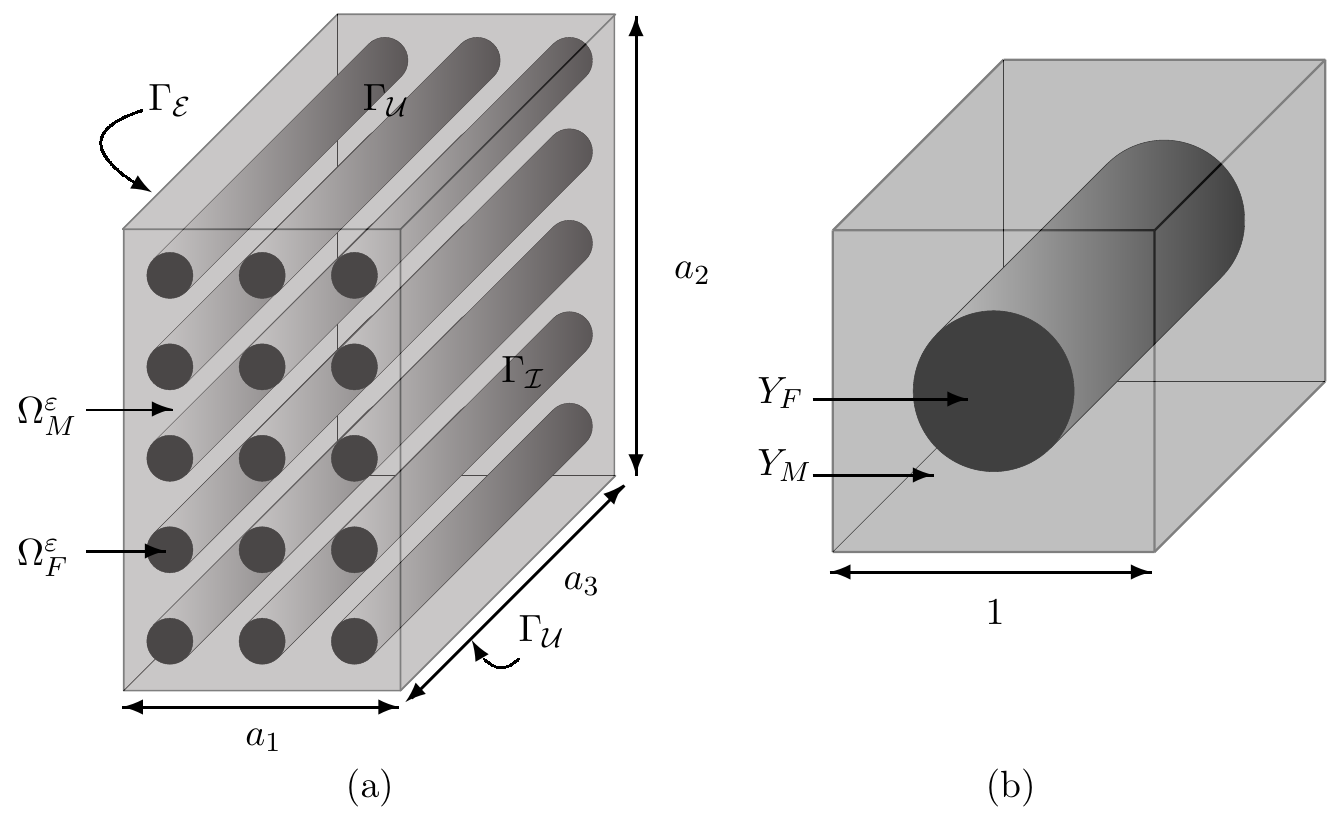}
\caption{(a) A depiction of the domain $\Omega$ with the subsets representing the cell wall matrix $\Omega_M^\ve$ and the microfibrils $\Omega_F^\ve$.  The surface $\Gamma_\cI$ is in contact with the interior of the cell, and the (hidden) surface $\Gamma_\cE$ is facing the outside of the cell, and $\Gamma_\cU$ is the union of the surfaces on the top and bottom of $\Omega$.  (b) A depiction of the unit cell $Y$.}\label{FigDomain1}
\end{figure}

We adopt the following  notation:
$
\Omega_T= (0,T) \times \Omega$,
 $\Omega_{M,T}^\ve= (0,T) \times \Omega_M^\ve$,
$\Gamma_{\cI,T}= (0,T) \times \Gamma_\cI$,   $\Gamma^\ve_{T}= (0,T) \times \Gamma^\ve$,  $\Gamma_{\cU,T}= (0,T) \times \Gamma_\cU$, 
$\Gamma_{\cE,T}= (0,T) \times \Gamma_\cE$, and  $\Gamma_{\cE\cU,T}= (0,T) \times \big(\Gamma_\cE\cup \Gamma_{\cU} \big)$,  and define 
\begin{equation*}
\begin{aligned}
\cW(\Omega)&=\{\bu\in H^{1}(\Omega;\Real^3)\ \big |\ \int_\Omega\bu \, dx=\textbf{0},\; \;  \int_\Omega[(\nabla\bu)_{12}- (\nabla \bu)_{21}]\, dx=\textbf{0}\  \text{and}\ \bu\ \text{is $a_3$-periodic $x_3$}\},\\
\mathcal V(\Omega_{M}^\ve)&=\{n\in  H^1(\Omega_M^\ve) \,  \big |\   
\; n\ \text{is $a_3$-periodic in  $x_3$}\}.
\end{aligned}
\end{equation*}
By Korn's second inequality,  the $L^2$-norm of the strain 
\begin{equation*}
\norm{\bu}_{\cW(\Omega)}=\norm{\be(\bu)}_{L^2(\Omega)}\rfa \bu\in\cW(\Omega)
\end{equation*}
defines a norm on $\cW(\Omega)$, see \cite{CC, Korn, OShY}.  For more details see also \cite{PS}.

The microscopic model for  elastic-viscoelastic deformations  $\bu^\ve$ of cell walls and the densities  of enzyme and pectins:  esterified pectin $\p^\ve_1$,  PME enzyme $\p^\ve_2$,  de-esterified pectin $\n^\ve_1$, calcium  ions $\n^\ve_2$, and calcium-pectin cross-links $b^\ve$ reads  
\begin{eqnarray}\label{visco2}
\left\{
\begin{aligned}
\text{div}(\mathbb E^\ve(n_b^\ve,x)\be(\bu^\ve) + \mathbb V^\ve(n_b^\ve,x)\partial_t\be(\bu^\ve))&=\textbf{0}&& \text{in}\ \Omega_T, \\
(\mathbb E^\ve(n_b^\ve,x)\be(\bu^\ve) +  \mathbb V^\ve(n_b^\ve,x)\partial_t\be(\bu^\ve))\bnu &=-p_\cI\bnu  && \text{on}\  \Gamma_{\cI,T}, \\
(\mathbb E^\ve(n_b^\ve,x)\be(\bu^\ve)+  \mathbb V^\ve(n_b^\ve,x)\partial_t\be(\bu^\ve))\bnu &=\bff && \text{on} \ \Gamma_{\cE\cU,T},\\
\bu^\ve &&&\text{$a_3$-periodic in }  x_3,\\
\bu^\ve(0,x)&=\bu_0(x) &&\text{in } \Omega,
\end{aligned}
\right.
\end{eqnarray}
and in the cell wall matrix $\Omega^\ve_{M,T}$ we consider  
\begin{equation}\label{reactions}
\begin{aligned}
\partial_t \p^\ve & =\text{div}(D_p\nabla \p^\ve) - \F_{p}(\p^\ve)\\
\partial_t \n^\ve & =\text{div}(D_n\nabla \n^\ve) + \F_{n}(\p^\ve, \n^\ve)
+  \R_n (\n^\ve, b^\ve, \mathcal N_\delta(\be(\bu^\ve))) \\
\partial_t b^\ve&= R_b(\n^\ve, b^\ve, \mathcal N_\delta(\be(\bu^\ve))), 
\end{aligned}
\end{equation}
where  $\p^\ve=(\p^\ve_1, \p^\ve_2)^T$, $\n^\ve=(\n^\ve_1, \n^\ve_2)^T$, ${\rm div}(D_p\nabla \p^\ve) =({\rm div}(D_p^1 \nabla \p^\ve_1), {\rm div}(D_p^2\nabla \p^\ve_2))^T$, and ${\rm div}(D_n\nabla \n^\ve) =({\rm div}(D_n^1 \nabla \n^\ve_1), {\rm div}(D_n^2\nabla \n^\ve_2))^T$,  together with the initial and  boundary conditions 
\beqn\label{BC}
\begin{aligned}
& D_p\nabla \p^\ve\, \bnu =  \J_p(\p^\ve)  && \text{on }\;  \Gamma_{\mathcal I, T}, \\
& D_p\nabla \p^\ve\, \bnu =  - \gamma_p \p^\ve  && \text{on } \; \Gamma_{\mathcal E, T}, \\
& D_n\nabla \n^\ve \, \bnu =  \mathcal N_\delta(\be(\bu^\ve))  \G (\n^\ve) && \text{on }\; \Gamma_{\cI, T},\\
& D_n\nabla \n^\ve \, \bnu = \J_n (\n^\ve) && \text{on }\; \Gamma_{\cE, T},\\
&D_p\nabla \p^\ve\, \bnu =0, \quad D_n\nabla \n^\ve\, \bnu =0,      && \text{on  }\; \Gamma^\ve_T \text{ and } \Gamma_{\mathcal U, T},  \\
& \p^\ve, \; \;  \n^\ve &&   \text{$a_3$-periodic in }  x_3,\\
&\p^\ve (0,x) = \p_{0}(x), \quad  \n^\ve (0,x) = \n_{0}(x), \quad b^\ve (0,x) = b_{0}(x) &&   \text{in  }\;   \Omega_M^\ve.
\end{aligned}
\eeqn 
 Here $\mathcal N_\delta(\be(\bu^\ve))$,  defined as
\begin{equation} \label{def_N_2}
\mathcal N_\delta(\be(\bu^\ve))= \Big( \dashint_{B_\delta \cap \Omega} {\rm tr} (\bbE^\ve(b^\ve) \be(\bu^\ve)) d\tilde x \Big)^+ \; \  \text{ in } \; \Omega_T, \quad \text{ for } \delta >0,
\end{equation}
represent the nonlocal impact of mechanical stresses on the calcium-pectin cross-links chemistry. 
From a biological point of view the non-local dependence of the chemical reactions on  the displacement  gradient is motivated by the fact that pectins are very long molecules and hence cell wall mechanics has a nonlocal impact on the chemical processes.
The positive part in the definition of ${\cN}_\delta(\be(\bu^\ve))$   reflects the fact that  extension rather than compression causes the breakage of cross-links.  In the boundary conditions \eqref{BC} we assumed that the flow of calcium ions between the interior of the cell  and the cell wall depends on the displacement gradient, which corresponds to the stress-dependent opening of calcium channels in the plasma  membrane \cite{White}. 

The elasticity and viscosity  tensors are defined as  $\mathbb E^\ve(\xi, x)=  \bbE(\xi, \hat x/\ve)$ and $\mathbb V^\ve(\xi, x)=  \mathbb V(\xi, \hat x/\ve)$,
  where the  $\hat Y$-periodic in $y$  functions $\mathbb E$ and $\mathbb V$ are given by
$
\bbE(\xi ,y)= \bbE_M(\xi) \chi_{\hat Y_M}(y) + \bbE_F \chi_{\hat Y_F}(y)
$
and 
 $\mathbb V(\xi ,y)= \mathbb V_{M} (\xi) \chi_{\hat Y_M}(y)$.

For a given  measurable set $\cA$ we use the notation
$
\langle \phi_1,\phi_2\rangle_{\cA}= \int_\cA \phi_1\phi_2\, dx,
$
where the product of $\phi_1$ and $\phi_2$  is the scalar-product if they are vector valued, and  by 
$\langle \psi_1,\psi_2\rangle_{\mathcal V, \mathcal V^\prime}$ we denote the dual product between  $\psi_1 \in L^2(0,T; \mathcal V(\Omega_{M}^\ve)) $ and  $\psi_2 \in L^2(0,T; \mathcal V(\Omega_{M}^\ve)^\prime)$. 
We also denote $\mathcal I^k_\mu = (-\mu, + \infty)^k$ for $\mu >0$ and $k \in \mathbb N$.

\begin{assumption}\label{assumptions}
\begin{itemize}
\item[1.]  $D_{\alpha}^j, D_b \in \mathbb R^{3\times 3}$ are symmetric,  with   $(D_{\alpha}^j \boldsymbol{\xi}, \boldsymbol{\xi})\geq d_\alpha |\boldsymbol{\xi}|^2$,  $(D_b \boldsymbol{\xi}, \boldsymbol{\xi}) \geq d_b|\boldsymbol{\xi}|^2$ for all $\boldsymbol{\xi} \in \mathbb R^3$ and some  $d_b, d_\alpha > 0$,  where $\alpha=p,n$, $j =1,2$, and  $\gamma_p , \gamma_b \geq 0$.
\item[2.]  $\F_{p}: \mathbb R^2 \to \mathbb R^2$  is continuously  differentiable in $\mathcal I_\mu^2$, with $\F_{p,1}(0, \eta)= 0$, $\F_{p,2}(\xi, 0)= 0$, $\F_{p,1}(\xi, \eta)\geq 0$, and  $|\F_{p,2}(\xi, \eta)|\leq g_1(\xi)(1+ \eta)$  for all $\xi, \eta \in \mathbb R_+$ and some $g_1 \in C^1(\mathbb R_+; \mathbb R_+)$.
\item[3.] $\J_p: \mathbb R^2\to \mathbb R^2$ is continuously  differentiable in $\mathcal I_\mu^2$,     with   $\J_{p,1}(0, \eta) \geq 0$, $\J_{p,2}(\xi, 0) \geq 0$, $|\J_{p,1}(\xi, \eta)|\leq \gamma_J(1+ \xi)$,  and $|\J_{p,2}(\xi, \eta)|\leq g(\xi)(1+ \eta)$   
for  all $\xi,\eta \in \mathbb R_+$ and some  $\gamma_J >0$ and   $g \in C^1(\mathbb R_+; \mathbb R_+)$.
\item[4.]  $\F_n: \mathbb R^4 \to \mathbb R^2$ is continuously  differentiable in $\mathcal I_\mu^4$, with  $ \F_{n,1}(\boldsymbol{\xi}, 0, \boldsymbol{\eta}_2)\geq 0$, $\F_{n,2}(\boldsymbol{\xi}, \boldsymbol{\eta}_1, 0)\geq 0$, and 
  \begin{equation*}
  \begin{aligned}
 |\F_{n,1}(\boldsymbol{\xi}, \boldsymbol{\eta})| \leq \gamma^1_F (1+ g_2(\boldsymbol{\xi})+|\boldsymbol{\eta}|), 
  && |\F_{n,2}(\boldsymbol{\xi}, \boldsymbol{\eta})| \leq \gamma^2_F (1+ g_2(\boldsymbol{\xi})+|\boldsymbol{\eta}|), 
  \end{aligned}
  \end{equation*}
  for  all $\boldsymbol{\xi}, \boldsymbol{\eta} \in \mathbb R^2_+$ and some $\gamma_F^1, \gamma_F^2  >0$ and  $g_2 \in C^1(\mathbb R_+^2; \mathbb R_+)$. 
\item[5.] $\R_n: \mathbb R^3\times \mathbb R_+ \to  \mathbb R^2$ and  $R_b:  \mathbb R^3\times \mathbb R_+ \to  \mathbb R$  are continuously differentiable in $\mathcal I_\mu^3 \times \mathbb R_+$ and satisfy 
  \begin{equation*}
  \begin{aligned}
&\R_{n,1}(0, \boldsymbol{\xi}_2, \eta,  \zeta) \geq 0, \quad &&   |\R_{n,1}(\boldsymbol{\xi},  \eta ,  \zeta)| \leq \beta_1(1+|\boldsymbol{\xi}| + \eta) (1+ \zeta),\\
&\R_{n,2}(\boldsymbol{\xi}_1, 0, \eta,  \zeta) \geq 0 ,    && |\R_{n,2}(\boldsymbol{\xi},  \eta ,  \zeta)| \leq \beta_2(1+|\boldsymbol{\xi}| + \eta) (1+ \zeta) ,   \\
&  R_b(\boldsymbol{\xi}, 0,  \zeta) \geq 0, && |R_b(\boldsymbol{\xi},  \eta ,  \zeta)| \leq \beta_3(1+|\boldsymbol{\xi}| + \eta) (1+\zeta), \qquad (R_b(\boldsymbol{\xi},  \eta ,  \zeta))^+\leq \beta_4
  \end{aligned}
  \end{equation*}
  for  some $\beta_j>0$, \, $j=1,\ldots, 4$,  and  all $\boldsymbol{\xi} \in \mathbb R^2_+$, $\eta, \zeta \in \mathbb R_+$.
\item[6.] $\J_n: \mathbb R^2\to \mathbb R^2$ is continuously  differentiable in $\mathcal I_\mu^2$,   with ${\J}_{n,1}(0, \eta) \geq 0$, ${\J}_{n,2}(\xi, 0) \geq 0$,    $|\J_{n,1} (\xi, \eta) |\leq \gamma_{n}^1(1+ \xi)$,  and $|\J_{n,2} (\xi, \eta)| \leq \gamma_{n}^2(1+ \xi+\eta)$   
for all $\xi, \eta \in \mathbb R_+$ and some $\gamma_{n}^1, \gamma_n^2 >0$. 
\item[7.]  $\G(\xi, \eta): \mathbb R^2 \to \mathbb R^2$, with  $\G(\xi, \eta)=(0,  \gamma_1  - \gamma_2 \eta)^T$ for $\eta \in \mathbb R$ and  some $\gamma_1, \gamma_2 \geq 0$. 
\item[8.]    $\mathbb V_M \in C^1(\mathbb R)$ possesses major and minor symmetries, i.e. $\mathbb V_{M, ijkl}= \mathbb V_{M, klij}=\mathbb V_{M, jikl}=\mathbb V_{M, ijlk}$,  and there exists $\omega_V>0$ such that   $\mathbb V_{M}(\xi)\bA \cdot \bA \geq \omega_V |\bA|^2$   for  all symmetric $\bA \in \mathbb R^{3\times 3}$  and $\xi \in \mathbb R_+$.
\item[9.] $\mathbb E_M \in C^1(\mathbb R)$,  $\mathbb E_F$, $\mathbb E_M$ possess major and minor symmetries, i.e. $\mathbb E_{L, ijkl}= \mathbb E_{L, klij}=\mathbb E_{L, jikl}=\mathbb E_{L, ijlk}$, for $L=F,M$, 
  and there exists $\omega_E >0$ such that  $\mathbb E_F \bA \cdot \bA \geq \omega_E |\bA|^2$ and  $\mathbb E_M(\xi) \bA \cdot \bA \geq \omega_E |\bA|^2$  for  all symmetric $\bA \in \mathbb R^{3\times 3}$ and $\xi \in \mathbb R_+$.
   There exists $\gamma_{M}>0$ such that  $|\mathbb E_M(\xi)| \leq \gamma_{M} $  for all  $\xi \in \mathbb R_+$.
   \item[10.] The initial conditions  $\p_0, \n_{0} \in  L^\infty(\Omega)^2$, $b_0 \in H^1(\Omega)\cap L^\infty(\Omega)$ are non-negative, and  $\bu_0 \in \cW(\Omega)$.
\item[11.] $\mathbf{f} \in H^1(0,T; L^2(\Gamma_{\cE}\cup \Gamma_{\cU}))^3$ and $p_{\cI} \in H^1(0,T; L^2(\Gamma_{\cI}))$.
\end{itemize} 
\end{assumption}  
{\bf Remark. } Notice that Assumption~\ref{assumptions}$.9$ is not restrictive from a physical point of view, since every biological material will have a maximal possible stiffness. Also, in contrast to \cite{PS}, we assume that  $(R_b(\boldsymbol{\xi},  \eta ,  \zeta))^+$ is bounded. This is required  to show a priori estimates for solutions of  equations of linear viscoelasticity independent of $b^\ve$. 

\begin{definition}
A weak solution of the microscopic model \eqref{visco2}--\eqref{BC} are functions $\p^\ve $, $\n^\ve$, and $b^\ve$  such that 
  $b^\ve \in H^1(0,T; L^2(\Omega^\ve_M))$, $\p^\ve, \n^\ve \in L^2(0,T;\cV(\Omega^\ve_M))$, $\partial_t \p^\ve, \partial_t \n^\ve \in L^2(0,T; \cV(\Omega^\ve_M)^\prime)$    and 
satisfy the equations 
\begin{eqnarray}\label{weak_sol_n1}
\begin{aligned}
\langle \partial_t \p^\ve, \bphi_p \rangle_{\cV, \cV^\prime} + \langle D_p \nabla \p^\ve, \nabla \bphi_p \rangle_{\Omega^\ve_{M,T}}& =-  \langle  \F_{p}(\p^\ve), \bphi_p \rangle_{\Omega^\ve_{M,T}}
+ \langle \J_p(\p^\ve), \bphi_p \rangle_{\Gamma_{\cI,T}}-  \langle \gamma_p \p^\ve, \bphi_p \rangle_{\Gamma_{\cE, T}},\\
\langle \partial_t \n^\ve, \bphi_n \rangle_{\cV, \cV^\prime}  + \langle D_n \nabla \n^\ve, \nabla \bphi_n \rangle_{\Omega^\ve_{M,T}}  &= \big\langle \F_{n}(\p^\ve, \n^\ve)+  \R_n(\n^\ve, b^\ve, \mathcal N_\delta(\be(\bu^\ve)))  , \bphi_n \big\rangle_{\Omega^\ve_{M,T}}\\
&+  \big\langle \mathcal N_\delta(\be(\bu^\ve)) \G(\n^\ve), \bphi_n \big \rangle_{\Gamma_{\cI,T}} +  \langle  \J_n  \n^\ve, \bphi_n \rangle_{\Gamma_{\cE, T}}
\end{aligned}
\end{eqnarray}
for all  $\bphi_p, \bphi_n\in L^2(0,T; \cV(\Omega_{M}^\ve))$,   
\begin{equation}\label{weak_sol_n2}
 \partial_t b^\ve =   R_b(\n^\ve, b^\ve,  \mathcal N_\delta(\be(\bu^\ve)))  \quad \text{a.e. in }  \Omega^\ve_{M,T}, 
\end{equation}
and 
 $\bu^\ve \in L^2(0,T; \cW(\Omega))$, with $\partial_t \be(\bu^\ve) \in L^2((0,T)\times\Omega_{M}^\ve)$,  satisfying 
\begin{equation}\label{weakvis2}
\big\langle  \mathbb E^\ve(b^\ve,x)\be(\bu^\ve) + \mathbb V^\ve(b^\ve,x)\partial_t \be(\bu^\ve), \be(\boldsymbol{\psi}) \big\rangle_{\Omega_T} =\langle \bff, \boldsymbol{\psi} \rangle_{\Gamma_{\cE\cU,T}} - \langle p_\cI \bnu, \boldsymbol{\psi} \rangle_{\Gamma_{\cI,T}}
\end{equation}
for all $\boldsymbol{\psi} \in L^2(0,T; \cW(\Omega))$.
Furthermore, $\p^\ve$, $\n^\ve$,  $b^\ve$   satisfy the  initial conditions  in $L^2(\Omega_M^\ve)$ and  $\bu^\ve$ satisfies the initial condition in $\cW(\Omega)$, i.e.\
$\bu^\ve(t,\cdot) \to \bu_0$ in $\cW(\Omega)$, 
 $\p^\ve(t,\cdot) \to \p_0$, 
 $\n^\ve(t,\cdot) \to \n_0$, $b^\ve(t,\cdot) \to b_0$ in $L^2(\Omega_M^\ve)$ as $t \to 0$. 
\end{definition}

\section{Main results}\label{mainresults}
The main result of the paper is   the derivation of the macroscopic equations for the microscopic viscoelastic model for plant cell wall biomechanics. The main difference between the homogenization results presented here and those  in \cite{PS} is due to the presence of degenerate viscose term in the equation for mechanical deformations  of a cell wall. The fact that only the cell wall matrix is viscoelastic and the dependence of the  viscosity tensor on the time variable,  via the dependence on the cross-links density $b^\ve$, make  the multiscale analysis nonclassical and complex. 
 
First we formulate the well-posedness  result for the model \eqref{visco2}--\eqref{BC}. 

\begin{theorem}\label{th_3}
Under  Assumptions \ref{assumptions}  there exists  a unique weak solution  of  \eqref{visco2}--\eqref{BC} satisfying the \textit{a priori} estimates
\begin{eqnarray}\label{apriori_estim_1}
 \| b^\ve \|_{L^\infty(0,T; L^\infty(\Omega_{M}^\ve))} + \|(\partial_t b^\ve)^+ \|_{L^\infty(0,T; L^\infty(\Omega_{M}^\ve))}\leq C_1,\qquad
\end{eqnarray}
where  the  constant $C_1$ is independent of $\ve$ and $\delta$, 
\begin{equation}\label{apriori_visco}
\|\bu^\ve\|_{L^\infty(0,T; \cW(\Omega))} + \|\partial_t  \be(\bu^\ve)\|_{L^2((0,T)\times\Omega_M^\ve)} \leq C_2,
\end{equation}
where the constant $C_2$ is independent of $\ve$, and 
\begin{equation}\label{apriori_estim}
\begin{aligned}
\|\partial_t b^\ve\|_{L^\infty(0,T; L^\infty(\Omega_{M}^\ve))} &\leq C_3, \\
\|\p^\ve\|_{L^\infty(0,T; L^\infty(\Omega_{M}^\ve))} + \|\nabla \p^\ve\|_{L^2( \Omega_{M,T}^\ve)} +\|\n^\ve\|_{L^\infty(0,T; L^\infty(\Omega_{M}^\ve))} + \|\nabla \n^\ve\|_{L^2( \Omega_{M,T}^\ve)} & \leq C_3,\qquad\\
\|\theta^h \p^\ve - \p^\ve\|_{L^2(\Omega_{M, T-h}^\ve)} + \|\theta^h \n^\ve- \n^\ve\|_{L^2(\Omega_{M, T-h}^\ve)}&\leq C_3 h^{1/4}
\end{aligned}
\end{equation}
for any $h>0$, where  $\theta^h v(t,x) = v(t+h, x)$ for $(t,x)\in \Omega_{M, T-h}^\ve$ and the constant $C_3$ is independent of $\ve$ and $h$.
\end{theorem}

The  proof of Thorem~\ref{th_3} follows similar lines  as  the proof of the corresponding existence and uniqueness results in \cite{PS}. Thus here we will  only  sketch  the main ideas of the proof and emphasise the steps that are different from those of  the proof in \cite{PS}.

To formulate the macroscopic equations for the microscopic model \eqref{visco2}--\eqref{BC}, first we define the macroscopic coefficients which will be obtained by the derivation of the limit equations. The macroscopic coefficients coming from the elasticity  tensor are given by
\begin{eqnarray}\label{macro_coef_3}
\begin{aligned}
 \widetilde{\mathbb E}_{{\rm{hom}}, ijkl}(b) &= \dashint_{\hat Y} \big[  \mathbb E_{ijkl}(b,y) + \big(\mathbb E(b,y) \hat \be_y(\mathbf{w}^{ij})\big)_{kl}  \big]   dy,\\
 \widetilde{\mathbb K}_{ijkl}(t,s,b) &= \dashint_{\hat Y}  \big(\mathbb E(b(t+s),y) \hat \be_y(\mathbf{v}^{ij}(t,s))\big)_{kl} dy,
  \end{aligned}
\end{eqnarray}
and the macroscopic elasticity and viscosity  tensors and memory kernel  read:
\begin{equation}\label{macro_coef_2}
\begin{aligned}
 \mathbb E_{{\rm{hom}}, ijkl}(b) &= \widetilde{\mathbb E}_{{\rm{hom}}, ijkl}(b) + \frac1{|\hat Y|}\int_{\hat Y_M}  \big(\mathbb V_M(b,y) \hat \be_y(\mathbf{w}^{ij}_t)\big)_{kl} dy, \\
  \mathbb V_{{\rm{hom}}, ijkl}(b) &=\frac1{|\hat Y|}\int_{\hat Y_M}  \big[  \mathbb V_{M, ijkl}(b,y) + \big(\mathbb V_M(b,y) \hat \be_y(\boldsymbol{\chi}^{ij}_{\mathbb V})\big)_{kl} \big] dy,\\
   \mathbb K_{ijkl}(t,s,b) &= \widetilde{\mathbb K}_{ijkl}(t,s,b) + \frac 1 {|\hat Y|} \int_{\hat Y_M} \big(\mathbb V_M(b(t+s),y) \hat \be_y(\mathbf{v}_t^{ij}(t,s))\big)_{kl}  dy,
  \end{aligned}
\end{equation}
where  $\mathbf{w}^{ij}$,   $\boldsymbol{\chi}^{ij}_{\mathbb V}$, and $\mathbf{v}^{ij}$ are solutions of the unit cell problems 
\begin{eqnarray}\label{unit_chi}
\begin{aligned}
\hat {\rm div}_{y} \left(\mathbb E(b,y) (\hat \be_y(\mathbf{w}^{ij})+ \textbf{b}_{ij}) + \mathbb V_M(b,y) \hat  \be_y(\mathbf{w}^{ij}_t) \chi_{\hat Y_M}\right)&={\bf 0}  \quad &&\text{in } \hat Y_T, \\
\mathbf{w}^{ij}(0,x,y) &={\bf 0} && \text{in }  \hat Y, \\  
\hat {\rm div}_{y} \big(\mathbb V_M(b,y)(\hat \be_y(\boldsymbol{\chi}^{ij}_{\mathbb V})+ \textbf{b}_{ij})\big) & ={\bf 0}  \quad &&\text{in } \hat Y_M, \\
\mathbb V_M(b,y)(\hat \be_y(\boldsymbol{\chi}^{ij}_{\mathbb V})+ \textbf{b}_{ij}) \bnu &={\bf 0}  \quad &&\text{on } \hat \Gamma, \\
\int_{\hat Y} \mathbf{w}^{ij} dy={\bf 0}, \quad  \int_{\hat Y_M} \boldsymbol{\chi}^{ij}_{\mathbb V}dy={\bf 0}, \qquad \qquad \mathbf{w}^{ij}, \; \; \boldsymbol{\chi}^{ij}_{\mathbb V} &    &&\hat Y\text{-periodic}, 
\end{aligned}
\end{eqnarray}
where $\bb_{jk} = \frac12 (\textbf{b}_j \otimes \textbf{b}_k + \textbf{b}_k \otimes \textbf{b}_j)$, with $(\textbf{b}_j)_{1\leq j \leq 3}$ being the canonical basis of $\mathbb R^3$,  and 
\begin{eqnarray}\label{unit_wv}
\begin{aligned}
&\hat {\rm div}_{y} \big(\mathbb E(b(t+s,x), y)  \hat \be_y(\mathbf{v}^{ij}) +  \mathbb V_M(b(t+s,x), y) \hat \be_y(\mathbf{v}^{ij}_t) \chi_{\hat Y_M} \big)= {\bf 0} \; \;  && \text{ in }  \hat Y_{T-s},\;\; \\
& \;\mathbf{v}^{ij}(0,s,x,y)= \overline{\boldsymbol{\chi}}^{ij}_{\mathbb V}(s,x,y) - \mathbf{w}^{ij}(s,x,y) &&   \text{ in  }  \hat Y, \\
&\; \int_{\hat Y}\mathbf{v}^{ij} dy={\bf 0}, \hspace{ 7 cm } \mathbf{v}^{ij} \quad && \hat Y\text{-periodic},
\end{aligned}
\end{eqnarray}
for a.a.\ $x \in \Omega$ and $s \in [0, T]$, where $\overline{\boldsymbol{\chi}}^{ij}_{\mathbb V}$ is an extension of 
$\boldsymbol{\chi}^{ij}_{\mathbb V}$ from $\hat Y_M$
 to $\hat Y$, such that $\int_{\hat Y} \overline{\boldsymbol{\chi}}^{ij}_{\mathbb V} dy ={\bf 0}$. Here   for a vector function  ${\bf v}$ we denote  $\hat{\rm div}_y  {\bf v} = \partial_{y_1} {\bf v}_1 +  \partial_{y_2} {\bf v}_2$. 
 
The macroscopic diffusion coefficients   are defined by    
\begin{equation}\label{macro_diff_coef}
\mathcal D^l_{\alpha, ij} = \dashint_{\hat Y_M} \big[  D^l_{\alpha, ij} + (D^l_{\alpha} \hat \nabla_{y}  v^j_{\alpha,l})_i \big] d \hat y \quad \text{ for }   i,j=1,2,3 \quad \alpha =p,n,
\end{equation}
 where $\hat\nabla_y v_{\alpha, l}^j=(\partial_{y_1} v_{\alpha, l}^j, \partial_{y_2} v_{\alpha, l}^j, 0)^T$
  and  the functions   $v_{\alpha, l}^{j}$, for $l=1,2$ and $j=1,2,3$,   are solutions of the unit cell problems
\begin{equation}\label{unit_n}
\begin{aligned}
&{\rm div}_{\hat y} (\hat D_\alpha^l \nabla_{\hat y} v^{j}_{\alpha, l}) =0 &&\text{ in } \hat Y_M, \quad j=1,2, 3,\\
&  (\hat D_\alpha^l \nabla_{\hat y} v^j_{\alpha,l} + \tilde D_\alpha^l \textbf{b}_j) \cdot \bnu =0  &&\text{ on } \hat \Gamma, \quad 
v^{j}_{\alpha, l}  \quad \hat Y -\text{ periodic}, \; \;  \int_{\hat Y_M} v^{j}_{\alpha,l} \, d y =0,
\end{aligned}
\end{equation}
where $\nabla_{\hat y} = (\partial_{y_1}, \partial_{y_2})$, $\hat D_\alpha^l = (D_{\alpha,ij}^l)_{i,j=1,2}$ and $ \tilde D_\alpha^l= (D_{\alpha,ij}^l)_{i=1,2, j=1,2,3}$.

Applying techniques of periodic homogenization we obtain  the macroscopic equations for plant cell wall biomechanics. 
\begin{theorem}\label{viscoelastic} 
A sequence of solutions of the microscopic model \eqref{visco2}--\eqref{BC} converges to a solution of the macroscopic equations
\begin{equation}\label{macro_1}
\begin{aligned}
  \partial_t  \p &= {\rm div} (\mathcal D_p \nabla \p) - \F_p (\p) ,   \\
\partial_t \n&= {\rm div} (\mathcal D_n \nabla \n) +  \F_n(\p, \n)+  \R_n(\n, b,  \mathcal N^{\rm eff}_\delta(\be(\bu))) ,  \\
\partial_t b &=  \phantom{{\rm div} (\mathcal D_c \nabla n_c)+  \F_n(\p, \n)+}  \;  R_b(\n, b,  \mathcal N^{\rm eff}_\delta (\be(\bu)))   
\end{aligned}
\end{equation}
in $\Omega_T$
   together with the initial and  boundary conditions 
\begin{equation}\label{macro_bc}
\begin{aligned}
& \mathcal D_p \nabla \p \,  \bnu = \theta_M^{-1} \J_p (\p),  &&  
\mathcal D_n \nabla \n \,  \bnu =\theta_M^{-1} \G(\n) \mathcal N^{\rm eff}_\delta(\be(\bu)) & & \text{on } \Gamma_{\cI,T}, \\
&\mathcal D_p \nabla \p \, \bnu = - \theta_M^{-1}\gamma_p  \, \p,    && \mathcal D_n \nabla \n \, \bnu =  \theta_M^{-1} \J_n \n   &&  \text{on } \Gamma_{\cE,T},  \\
&\mathcal D_p \nabla \p \,  \bnu = 0,  \quad  &&  \mathcal D_n \nabla \n \,  \bnu = 0 && \text{on } \Gamma_{\cU,T}, \\
&\p, \quad \n && a_3-\text{periodic in } x_3, \\
& \p (0) = \p_0 (x), \quad \n(0)= \n_0, \quad && b(0)= b_0 &&\text{in }  \Omega,  
\end{aligned}
\end{equation}
 where $\theta_M= |\hat Y_M|/ |\hat Y|$, and the macroscopic equations of linear viscoelasticity 
\begin{eqnarray}\label{macro_vis}
\begin{aligned}
{\rm div}\Big( \mathbb E_{\rm{ hom}}\be(\bu)  +  \mathbb V_{\rm{hom}}\partial_t \be(\bu) + \int_0^t   \mathbb K(t-s,s) \partial_s \be(\bu) \, ds \Big)&={\bf 0}  \; \;\; \;   \quad \text{ in }  \Omega_T, \\
\big(\mathbb E_{\rm{hom}}\be(\bu)  +  \mathbb V_{\rm{hom}}\partial_t \be(\bu) + \int_0^t  \mathbb K(t-s,s)\partial_s \be(\bu) \,  ds \big)\bnu &= \bff \;\; \;  \quad  \text{ on } \Gamma_{\cE\cU,T},\quad\quad \\
\big(\mathbb E_{\rm{hom}}\be(\bu)  +  \mathbb V_{\rm{hom}}\partial_t \be(\bu) + \int_0^t  \mathbb K(t-s,s) \partial_s\be(\bu) \, ds\big) \bnu& = - p_\cI \bnu \;    \text{ on } \Gamma_{\cI,T},\quad \\
\bu(0,x) &= \bu_0(x) \; \;    \text{ in } \Omega.
\end{aligned}
\end{eqnarray}
\end{theorem} 
Here  
 \begin{eqnarray}\label{macro_N_d_visco}
 && \;  \mathcal N^{\rm eff}_\delta(\be(\bu))= \Big(\dashint_{B_\delta(x)\cap \Omega} {\rm tr}
\Big[\widetilde{\mathbb E}_{\rm{hom}} (b) \be(\bu)  + \int_0^t \widetilde{\mathbb K}(t-s,s, b) \partial_s \be(\bu) ds \Big] d\tilde x \Big)^+ \;  \text{ for all } (t,x)\in (0,T)\times \overline \Omega. 
\end{eqnarray}

\section{Existence of a unique weak solution of the microscopic problem \eqref{visco2}--\eqref{BC}. A priori estimates.}\label{existence} 

In the derivation of a priori estimates for solutions of the microscopic problem  \eqref{visco2}--\eqref{BC}  we shall use an extension of a function defined on a connected perforated domain $\Omega^\ve_M$ to $\Omega$. Applying classical extension results \cite{Acerbi,CiorPaulin99, MP}, we obtain the following lemma.
\begin{lemma}\label{lem:extension}
There exists an extension $\overline v^{\varepsilon}$ of $v^{\varepsilon}$ from
$W^{1,p}(\Omega_M^{\varepsilon})$ into $W^{1,p}(\Omega)$, with  $1\leq p< \infty$, such that 
\[
\|\overline v^{\varepsilon}\|_{L^{p}(\Omega)}\leq\mu_1\|v^{\varepsilon}\|_{L^{p}(\Omega_M^{\varepsilon})} \; \text{ and } \; 
\|\nabla\overline v^{\varepsilon}\|_{L^{p}(\Omega)}\leq\mu_1\|\nabla v^{\varepsilon}\|_{L^{p}(\Omega_M^{\varepsilon})},
\]
where   the constant $\mu_1$ depends only on $Y$ and $Y_M$, and $Y_M\subset Y$ is connected. 
\end{lemma}

\noindent\textbf{Remark. } Notice that the microfibrils do not intersect the boundaries $\Gamma_\cI$, $\Gamma_\cU$,  and $\Gamma_{\cE}$, and near the boundaries $\partial \Omega \setminus(\Gamma_{\cI} \cup \Gamma_{\cE}\cup \Gamma_\cU)$ it is sufficient to extend $v^{\varepsilon}$ by reflection in the direction normal to the microfibrils and parallel to the boundary. Thus, classical extension results~\cite{Acerbi,CiorPaulin99, JH1991, MP} apply to $\Omega^\ve_M$.
%
\noindent In the sequel, we identify $\p^{\varepsilon}$ and $\n^\ve$ with their extensions.

First we show the well-possedness and a priori estimates for equations  \eqref{reactions}--\eqref{BC}  for a given $\bu^\ve \in L^\infty(0,T; \cW(\Omega))$. Next for a given $b^\ve$ we show the existence of a unique solution of the viscoelastic problem \eqref{visco2}. Then using the fact that the estimates for $b^\ve$ can be obtain   independently of $\bu^\ve$ and applying a fixed point argument we show the well-possedness of the coupled  system. 

\begin{lemma}\label{th:exist_1}
Under Assumption~\ref{assumptions} and for $\bu^\ve \in L^\infty(0,T; \cW(\Omega))$ such that 
\begin{equation}\label{estim_u_1}
\| \bu^\ve\|_{L^\infty(0,T;\cW(\Omega))} \leq C,
\end{equation}
where the constant $C$ is  independent of  $\ve$, 
 there exists a unique weak solution $(\p^\ve, \n^\ve, b^\ve)$ of the microscopic model \eqref{reactions}--\eqref{BC} satisfying 
$$ \p^\ve_j(t,x)\geq 0,  \; \; \n^\ve_j(t,x)\geq 0, \; \; b^\ve(t,x)\geq 0\qquad \text{ for a.a.}\  (t,x) \in (0,T)\times\Omega_M^\ve,  \quad j=1,2,$$
and the \textit{a priori} estimates
   \eqref{apriori_estim_1} and \eqref{apriori_estim}.
\end{lemma}
\begin{proof}  The proof of this lemma  follows the same lines  as the proof of Theorem~3.3  in \cite{PS}.
The only difference is in the derivation of the estimates for $b^\ve$.  Using the non-negativity  of $\n_1^\ve$, $\n_2^\ve$, $b^\ve$,  and Assumptions~\ref{assumptions}.4~and~\ref{assumptions}.5 we obtain from the equation for $b^\ve$ 
\begin{equation}\label{bound_n_b}
\begin{aligned}
0 \leq b^\ve(t,x) &\leq  \|b_0\|_{L^\infty(\Omega)} + T \|(R_{b}(\n^\ve, b^\ve, \mathcal N_\delta(\be(\bu^\ve))))^+\|_{L^\infty(\Omega_{M,T}^\ve)} \leq C   && \text{for a.a. } (t,x) \in \Omega_{M,T}^\ve, \\
(\partial_t b^\ve (t,x))^{+}& \leq   \|(R_{b}(\n^\ve, b^\ve, \mathcal N_\delta(\be(\bu^\ve))))^+\|_{L^\infty(\Omega_{M,T}^\ve)} \leq \beta_4   && \text{for a.a. } (t,x) \in \Omega_{M,T}^\ve.
\end{aligned}
\end{equation}
Hence, the bounds for $ b^\ve$ and $(\partial_t b^\ve)^+$ are independent of the bound for  $\|\bu^\ve\|_{L^\infty(0,T; \mathcal{W}(\Omega))}$. This fact is important for the derivation of a priori estimates for $\bu^\ve$ and the fixed point argument for the proof of the existence of a solution for the coupled system. 

Using the equation for $b^\ve$, the definition of $\mathcal N_\delta$ and the estimates for $\|\n^\ve\|_{L^\infty(0,T; L^\infty(\Omega_M^\ve))}$, $\|b^\ve\|_{L^\infty(0,T; L^\infty(\Omega_M^\ve))}$,  and  $\| \bu^\ve\|_{L^\infty(0,T;\cW(\Omega))}$  we obtain the estimate for $\|\partial_t b^\ve\|_{L^\infty(0,T; L^\infty(\Omega_{M}^\ve))}$ uniformly in $\ve$.

Similar to  \cite{PS},  considering $\boldsymbol{\phi}_p=\int_t^{t+h} [\theta_h \p^\ve(s,x) - \p^\ve(s,x)] ds $ and $\bphi_n=\int_t^{t+h} [\theta_h \n^\ve(s,x) - \n^\ve(s,x) ]ds $ as test functions in \eqref{weak_sol_n1}, respectively, we obtain the last estimate in  \eqref{apriori_estim}. 
\end{proof}

%

Next we prove the existence, uniqueness and a priori estimates for  a solution of viscoelastic  equations  for a given   ${b}^\ve\in L^\infty(0,T; L^\infty(\Omega_M^\ve))$.

\begin{lemma}\label{th_2}
Under  Assumption~\ref{assumptions}  for a given ${b}^\ve\in L^\infty(0,T; L^\infty(\Omega_M^\ve))$, satisfying  
$$\| b^\ve\|_{L^\infty(0,T; L^\infty(\Omega_M^\ve)) }\ + \| ( \partial_t {b}^\ve)^{+} \|_{L^\infty(0,T; L^\infty(\Omega_M^\ve)) }\leq C, $$  there exists  a weak solution  of the degenerate viscoelastic equations \eqref{visco2} satisfying the \textit{a priori} estimates  \eqref{apriori_visco}.
\end{lemma}
\begin{proof}
Using the estimates for $\bu^\ve$ and $\partial_t\bu^\ve$, similar to those in \eqref{estim_u_t_vis},  along with the  positive definiteness of $\mathbb E$ and $\mathbb V$, and applying the Galerkin method, yield the existence of  a weak solution of the problem~\eqref{visco2}. 

Considering $\partial_t \bu^\ve$ as a test function in \eqref{weakvis2} and using the non-negativity of $b^\ve$ and the assumptions on $\mathbb E$ and $\mathbb V$,  we obtain
\begin{eqnarray*}
\| \be(\bu^\ve)(\tau)\|^2_{L^2(\Omega)}  +   \|\partial_t \be(\bu^\ve) \|^2_{L^2(\Omega^\ve_{M,\tau})}  \leq  \langle (\partial_t b^\ve)^+ \mathbb E^\prime_{M}(b^\ve) \be(\bu^\ve), \be(\bu^\ve)  \rangle_{\Omega_{M,\tau}^\ve} + C_1 \| \be(\bu_0)\|^2_{L^2(\Omega)} \\
+  \langle \bff, \partial_t \bu^\ve \rangle_{\Gamma_{\cE,\tau}}   -   \langle p_{\cI} \bnu, \partial_t \bu^\ve \rangle_{\Gamma_{\cI,\tau}} 
\leq   \sigma  \| \be(\bu^\ve)(\tau) \|^2_{L^2(\Omega)}
 + C_2  \| \be(\bu^\ve) \|^2_{L^2(\Omega_{\tau})} 
 + C_\sigma \big[\| \partial_t \bff \|^2_{L^2(\Gamma_{\cE,\tau})}\\ + \| \partial_t p_\cI \|^2_{L^2(\Gamma_{\cI,\tau})}   +\|\bff(\tau) \|^2_{L^2(\Gamma_{\cE})} + \| p_\cI(\tau) \|^2_{L^2(\Gamma_{\cI})}+
  \|\bff(0) \|^2_{L^2(\Gamma_{\cE})} + \| p_\cI(0) \|^2_{L^2(\Gamma_{\cI})}\big]+ C_3
\end{eqnarray*}
for  $\tau \in [0,T]$. Choosing $\sigma$ sufficiently small, using    the  boundedness of $b^\ve$ and $(\partial_t b^\ve)^+$, independent of  $\ve$ and $\bu^\ve$,   and applying  Gronwall's inequality imply
\begin{equation}\label{estim_u_t_vis}
\| \be(\bu^\ve)\|^2_{L^\infty(0,T;L^2(\Omega))}  +\| \partial_t\be(\bu^\ve)\|^2_{L^2(\Omega_{M,T}^\ve)}   \leq C,
\end{equation}
with a constant $C$ independent of  $\ve$.  Then using  the second Korn inequality yields \eqref{apriori_visco}.
\end{proof}

Now applying a fixed point argument and using the results in Lemmas~\ref{th:exist_1}~and~\ref{th_2} we obtain the well-possedness result for the coupled system \eqref{visco2}--\eqref{BC}. 

\begin{proof}[Proof of Theorem~\ref{th_3}]
We have that for a given $\widetilde \bu^\ve \in L^\infty(0,T;\cW(\Omega))$, with $\|\widetilde \bu^\ve \|_{L^\infty(0,T;\cW(\Omega))} \leq C$,  Lemma~\ref{th:exist_1} implies the existence of a non-negative weak solution $(\p^\ve, \n^\ve, b^\ve)$  of the problem  \eqref{reactions}--\eqref{BC}, where the estimates for $\|b^\ve\|_{L^\infty(0,T; L^\infty(\Omega_M^\ve))}$ and $\|(\partial_t b^\ve)^{+}\|_{L^\infty(0,T; L^\infty(\Omega_M^\ve))}$ are independent of $\widetilde \bu^\ve$.  Then for $b^\ve$  from Lemma~\ref{th_2} we have a solution $\bu^\ve$ of \eqref{visco2}.

We define $\mathcal K: L^\infty(0,  T; \cW(\Omega))\to   L^\infty(0,  T, \cW(\Omega))$ by $\mathcal K(\widetilde \bu^\ve)=\bu^\ve$, where $\bu^\ve$ is a solution of  \eqref{visco2}  for $b^\ve$ which is a solution of   \eqref{reactions}--\eqref{BC} with $\widetilde \bu^\ve$ instead of  $\bu^\ve$, and  show that for sufficiently small $\widetilde T \in (0,T]$, the operator $\mathcal K: L^\infty(0,  \widetilde T; \cW(\Omega))\to   L^\infty(0,  \widetilde T, \cW(\Omega))$ is a contraction, i.e.\
$$
\|\mathcal K(\widetilde \bu_1^\ve) - \mathcal K(\widetilde \bu_2^\ve)\|_{ L^\infty(0,  \widetilde T; \cW(\Omega))} 
\leq \gamma \| \bu_1^\ve -  \bu_2^\ve\|_{L^\infty(0,  \widetilde T; \cW(\Omega))}, \quad \text{ for some } \; \; 0< \gamma <1.
$$

 Considering the difference of  equation  \eqref{weakvis2} for $b^{\ve, 1}$ and $b^{\ve, 2}$,  and taking  $\partial_t(\bu^{\ve,1}-\bu^{\ve,2})$ as a test function  yield 
\begin{equation*}
\begin{aligned}
\langle \bbE^\ve(b^{\ve,1},x)\be(\bu^{\ve,1}-\bu^{\ve,2}), \partial_t\be(\bu^{\ve,1}-\bu^{\ve,2})\rangle_{\Omega} + 
\langle \mathbb V^\ve(b^\ve_{1},x)\partial_t\be(\bu^{\ve,1}-\bu^{\ve,2}),\partial_t \be(\bu^{\ve,1}-\bu^{\ve,2})\rangle_{\Omega}
\\=\langle(\bbE^\ve_M(b^{\ve,1},x)-\bbE^\ve_M(b^{\ve,2},x))\be(\bu^\ve_2),\partial_t  \be(\bu^{\ve,1}-\bu^{\ve,2})\rangle_{\Omega_M^\ve}
\\+ 
\langle(\mathbb V^\ve_M(b^{\ve,1},x)-\mathbb V^\ve_M(b^{\ve, 2},x))\partial_t\be(\bu^{\ve,2}), \partial_t\be(\bu^{\ve,1}-\bu^{\ve,2})\rangle_{\Omega_M^\ve}
\end{aligned}
\end{equation*}
for $t\in (0,T]$.  By the assumptions  on $\bbE^\ve(b^{\ve,1},x)$ and $\mathbb V^\ve(b^{\ve,1},x)$, we have
\begin{align*}
\|\be(\bu^{\ve,1}(\tau))-\be(\bu^{\ve,2}(\tau))\|_{L^2(\Omega)}^2 &\leq
  C_1 \|(\partial_t b^{\ve,1})^+\|_{L^\infty(0, T; L^\infty(\Omega^\ve_{M}))}\int_0^\tau \|\be(\bu^{\ve,1}-\bu^{\ve,2})\|_{L^2(\Omega_M^\ve)}^2   d\tau
\\ & \qquad
+  C_2 \|\be(\bu^{\ve,2})\|^2_{H^1(0,T; L^2(\Omega_M^\ve))}\|b^{\ve,1}- b^{\ve,2}\|_{ L^\infty(0,\tau; L^\infty(\Omega^\ve_{M}))}^2.
\end{align*}
Applying the Gronwall inequality and the estimates for $\partial_t b^{\ve,1}$ and $\be(\bu^{\ve,2})$ implies
$$
\|\be(\bu^{\ve,1})-\be(\bu^{\ve,2})\|_{L^\infty(0,\tilde T; L^2 (\Omega))}^2 \leq C_3 \|b^{\ve,1}- b^{\ve, 2}\|_{L^\infty(0, \widetilde T; L^\infty(\Omega^\ve_{M}))}^2
$$
for  $\widetilde T \in (0, T]$. 

Now we shall estimate   $\|b^{\ve,1}- b^{\ve,2}\|_{L^\infty(0, \widetilde T; L^\infty(\Omega^\ve_{M}))}$  in terms of 
$\widetilde T\|\be(\bu^{\ve,1})-\be(\bu^{\ve,2})\|_{L^\infty(0,\widetilde T; L^2 (\Omega))}$. 
Following the same calculations as in  \cite{PS}, by taking $\bphi_n= |\n^{\ve,1} - \n^{\ve,2}|^{p-2}(\n^{\ve,1} - \n^{\ve,2})$, where $p =2^\kappa$ and $\kappa=1,2,3,\ldots$,  as test functions in the  differences of the equations for $\n^{\ve,1}$ and $\n^{\ve,2}$,   and applying iterations in $p$ similar to  Lemma~3.2 in Alikakos \cite{Alikakos}, we obtain  
\begin{equation}\label{estim_contr}
\|b^{\ve,1}-b^{\ve,2}\|_{L^\infty(0, \widetilde T; L^\infty(\Omega^\ve_M))}^2 \leq C_4 \widetilde T\|\be (\widetilde \bu^{\ve,1}-  \widetilde \bu^{\ve,2})\|_{L^\infty(0,\widetilde T; L^2(\Omega))}^2
\end{equation}
for  $\widetilde T \in (0,T]$.  
Thus, we have that the operator $\mathcal K: L^\infty(0,\widetilde T; \cW(\Omega)) \to L^\infty(0, \widetilde T; \cW(\Omega))$, defined by $\mathcal K(\widetilde \bu^\ve) = \bu^\ve$, where $\bu^\ve$ is a weak solution of   \eqref{visco2},  is a contraction for sufficiently small $\tilde T$, where $\widetilde T$ depend on the coefficients in the equations and is independent of $(\p^\ve, \n^\ve, b^\ve, \bu^\ve)$. 
Hence, using the Banach fixed point theorem and  iterating over time intervals, we obtain the existence of a unique weak solution of the microscopic problem~\eqref{visco2}--\eqref{BC}. 
\end{proof}

\section{Derivation of the macroscopic equations of the problem \eqref{visco2}-\eqref{BC}:  Proof of  Theorem \ref{viscoelastic}.}\label{secthom}

%

Due to the fact that viscous term is defined only in the cell wall matrix and is zero for cell wall microfibrils, to conduct the multiscale analysis of the viscoelastic problem \eqref{visco2} we first consider  a perturbed problem by adding the inertial term 
$\vartheta \partial_t^2 \bu^\ve$, where $\vartheta>0$ is a small perturbation parameter:
\begin{equation}\label{perturb_eq}
\vartheta \chi_{\Omega_M^\ve} \partial_t^2 \bu^\ve=\text{div}(\mathbb E^\ve(b^\ve,x)\be(\bu^\ve) + \mathbb V^\ve(b^\ve,x)\partial_t\be(\bu^\ve)) \; \; \quad \text{on}\;  \Omega_T,
\end{equation}
 and the additional  initial condition 
 \begin{equation}\label{in_perturbed}
 \partial_t \bu^\ve (0,x)={\bf 0} \quad \text{  in } \quad  \Omega.
 \end{equation} 
 
 We split the proof of  Theorem \ref{viscoelastic} into two steps. First we derive the macroscopic equations for the perturbed system. Then letting the perturbation parameter $\vartheta$  go to  zero we obtain the macroscopic equations \eqref{macro_vis} for the original degenerate viscoelastic problem.

 \begin{lemma} \label{lemma_exist}
 There exists a unique solution of the perturbed microscopic problem \eqref{reactions}, \eqref{BC} and  \eqref{perturb_eq}, together with the initial and boundary conditions in \eqref{visco2} and \eqref{in_perturbed}, satisfying the {\it a priori} estimates 
 \begin{eqnarray}\label{apriori_visco_2}
\begin{aligned}
& \vartheta^{\frac 12}\|\partial_t \bu^\ve\|_{L^\infty(0,T; L^2(\Omega_M^\ve))}+  \|\bu^\ve\|_{L^\infty(0,T; \cW(\Omega))}  + \|\partial_t \be(\bu^\ve)\|_{L^2(\Omega_{M,T}^\ve)} \leq C,
  \end{aligned}
\end{eqnarray}
 \begin{eqnarray}\label{apriori_visco_3}
\begin{aligned}
&\|\p^\ve\|_{L^\infty(0,T; L^\infty(\Omega_{M}^\ve))} + \|\nabla \p^\ve\|_{L^2(\Omega_{M,T}^\ve)}+ \|\n^\ve\|_{L^\infty(0,T; L^\infty(\Omega_{M}^\ve))} + \|\nabla \n^\ve\|_{L^2(\Omega_{M,T}^\ve)} \leq C,  \\
& \|b^\ve\|_{L^\infty(0,T;L^\infty(\Omega_{M}^\ve))}  +   \|\partial_t b^\ve\|_{L^\infty(0,T; L^\infty(\Omega_{M}^\ve))}\leq C,  
  \end{aligned}
\end{eqnarray}
 with  a constant $C$ independent of $\ve$ and $\vartheta$, and 
 \begin{eqnarray}\label{apriori_visco_4}
\begin{aligned}
&  \|\theta_h \p^\ve - \p^\ve\|_{L^2(\Omega_{M,T-h}^\ve)}+ \|\theta_h \n^\ve - \n^\ve\|_{L^2(\Omega_{M,T-h}^\ve)} \leq Ch^{1/4}, 
  \end{aligned}
\end{eqnarray}
where  $\theta_h v(t,x)=v(t+h, x)$ for a.e.\ $(t,x)\in \Omega_{M,T-h}^\ve$, and the constant  $C$ is independent of $\ve$ and $\vartheta$.
 \end{lemma} 
 
 \begin{proof}   
 For a given $\bu^\ve \in L^\infty(0,T; \cW(\Omega))$, with $\|\bu^\ve\|_{L^\infty(0,T; \cW(\Omega))}\leq C$,
  in the same way as in Lemma~\ref{th:exist_1}  we obtain the existence of a unique solution of the  problem \eqref{reactions}--\eqref{BC},   satisfying the  a priori  estimates \eqref{apriori_visco_3}.  Notice that the estimates for $b^\ve$ and $(\partial_t  b^\ve)^{+} $ are independent of  $\bu^\ve$, $\ve$,  and $\vartheta$.
  
  Then for  $b^\ve \in L^\infty(0,T; L^\infty(\Omega_M^\ve))$, with $\| b^\ve \|_{L^\infty(0,T; L^\infty(\Omega_M^\ve))}\leq C$ and  $ |(\partial_t b^\ve)^{+} \|_{L^\infty(0,T; L^\infty(\Omega_M^\ve))} \leq C$, similar to Lemma~\ref{th_2}, we obtain the existence of a weak solution of the perturbed equations \eqref{perturb_eq} with initial and boundary conditions in   \eqref{visco2} and \eqref{in_perturbed},  satisfying the   a priori estimates \eqref{apriori_visco_2}. 

Similar to the proof of Theorem~\ref{th_3},   considering the difference of the equations \eqref{perturb_eq} for $b^{\ve, j}$, with $j=1,2$,  and taking  $\partial_t(\bu^{\ve,1}-\bu^{\ve,2})$ as a test function  yield 
\begin{equation*}
\begin{aligned}
&\frac 12  \vartheta \|\partial_t (\bu^{\ve,1}(\tau)- \bu^{\ve,2}(\tau))\|^2_{L^2(\Omega_M^\ve)} +  \langle \bbE^\ve(b^{\ve,1},x)\be(\bu^{\ve,1}-\bu^{\ve,2}), \partial_t\be(\bu^{\ve,1}-\bu^{\ve,2})\rangle_{\Omega_\tau} \\
& + 
\langle \mathbb V^\ve(b^{\ve,1},x)\partial_t\be(\bu^{\ve,1}-\bu^{\ve,2}),\partial_t \be(\bu^{\ve,1}-\bu^{\ve,2})\rangle_{\Omega_\tau}
=\langle(\bbE^\ve_M(b^{\ve,1},x)-\bbE^\ve_M(b^{\ve,2},x))\be(\bu^{\ve, 2}),\partial_t  \be(\bu^{\ve,1}-\bu^{\ve, 2})\rangle_{\Omega_{M,\tau}^\ve}
\\
& + 
\langle(\mathbb V^\ve_M(b^{\ve,1},x)-\mathbb V^\ve_M(b^{\ve,2},x))\partial_t\be(\bu^{\ve,2}), \partial_t\be(\bu^{\ve,1}-\bu^{\ve,2})\rangle_{\Omega_{M,\tau}^\ve}
\end{aligned}
\end{equation*}
for  $\tau\in (0,T]$.  By the assumptions  on $\bbE^\ve(b^{\ve,1},x)$ and $\mathbb V^\ve(b^{\ve,1},x)$, 
and 
applying the Gronwall inequality and the estimates for $\partial_t b^{\ve,1}$ and $\be(\bu^{\ve, 2})$ we obtain 
\begin{equation}\label{estim_diff_pert}
\|\be(\bu^{\ve, 1})-\be(\bu^{\ve, 2})\|_{L^\infty(0,\widetilde T; L^2 (\Omega))}^2 \leq C_3 \|b^{\ve,1}- b^{\ve,2}\|_{L^\infty(0, \widetilde T; L^\infty(\Omega^\ve_{M}))}^2
\end{equation}
for all $\widetilde T \in (0, T]$. 
Then, using the estimates \eqref{estim_contr},  \eqref{apriori_visco_2} and  \eqref{estim_diff_pert}, and the a priori estimates for $\p^\ve$, $\n^\ve$, and $b^\ve$ in the same way as in the proof of Theorem~\ref{th_3}  we obtain the existence of a unique weak  solution of the perturbed problem \eqref{reactions}, \eqref{BC}, and \eqref{perturb_eq} with  initial and boundary conditions in \eqref{visco2}  and   \eqref{in_perturbed}. 
\end{proof}

\begin{lemma}\label{converg_vartheta}
There exist functions $\p^\vartheta, \n^\vartheta \in L^2(0,T; \cV(\Omega))\cap L^\infty(0,T; L^\infty(\Omega))^2$, $\hat\p^\vartheta, \hat \n^\vartheta \in L^2(\Omega_T; H^1_{\rm per}(\hat Y) / \mathbb R)^2$ and $b^\vartheta\in W^{1,\infty}(0,T; L^\infty(\Omega))$, $\bu^\vartheta \in H^1(0,T; \cW(\Omega))$, $\hat \bu^\vartheta \in L^2(\Omega_T; H^1_{\rm per}(\hat Y) / \mathbb R)^3$,  $\partial_t \hat \bu^\vartheta \in L^2(\Omega_T; H^1_{\rm per}(\hat Y_M) / \mathbb R)^3$ such that for a subsequence of solutions  $(\p^\ve, \n^\ve, b^\ve, \bu^\ve)$ 
of the microscopic problem  \eqref{reactions}, \eqref{BC}, and \eqref{perturb_eq},  with initial and boundary conditions in  \eqref{visco2} and \eqref{in_perturbed},  $($denoted again by $(\p^\ve, \n^\ve, b^\ve, \bu^\ve))$ we have the following convergence results: 

\begin{equation}
\begin{aligned}
&\p^\ve \rightharpoonup \p^\vartheta, \quad \n^\ve \rightharpoonup \n^\vartheta && \text{ weakly in } L^2(0,T; H^1(\Omega)), \\
&\p^\ve \to \p^\vartheta, \quad \n^\ve \to \n^\vartheta && \text{ strongly in } L^2(\Omega_T), \\
&\nabla \p^\ve \rightharpoonup \nabla \p^\vartheta +\hat \nabla_{ y}  \hat \p^\vartheta, \quad \nabla \n^\ve \rightharpoonup \nabla \n^\vartheta +\hat  \nabla_{y} \hat \n^\vartheta  && \text{ two-scale}, \\
&b^\ve \rightharpoonup b^\vartheta, \quad \partial_t b^\ve \rightharpoonup \partial_t b^\vartheta  &&  \text{ two-scale}, \\
&\mathcal T^\ast_\ve(b^\ve) \to b^\vartheta \quad && \text{ strongly in } L^2(\Omega \times Y_M) , \\
& \bu^\ve \rightharpoonup \bu^\vartheta && \text{ weakly in } L^2(0,T; \cW(\Omega)), \\
&\nabla  \bu^\ve \rightharpoonup \nabla \bu^\vartheta +\hat  \nabla_{y}  \hat \bu^\vartheta && \text{ two-scale}, \\
&\chi_{\Omega_M^\ve} \nabla \partial_t \bu^\ve \rightharpoonup \chi_{\hat Y_M}(\nabla \partial_t \bu^\vartheta +\hat  \nabla_{y} \partial_t  \hat \bu^\vartheta) 
&& \text{ two-scale}.
\end{aligned}
\end{equation}
\end{lemma}
Here $\cT^\ast_\ve: L^p(\Omega_{M,T}^\ve) \to L^p(\Omega_T \times \hat Y_M)$ is the unfolding operator defined as $\mathcal T^\ast_\ve (\phi)(t,x,y) = \phi(t,\ve [\hat x/\ve]_{\hat Y_M} + \ve y, x_3)$ for $(t,x)\in \Omega_T$ and $y \in \hat Y_M$, where $\hat x = (x_1, x_2)$ and $[\hat x/\ve]_{\hat Y_M}$ is the unique integer combination of the periods such that $\hat x/\ve - [\hat x/\ve]_{Y_M} \in \hat Y_M $, see e.g.\ \cite{CDDGZ}. 

\begin{proof}
A priori estimates in \eqref{apriori_estim_1} and \eqref{apriori_estim}   imply weak and two-scale convergences of $\p^\ve$, $\n^\ve$, $b^\ve$, and $\partial_t b^\ve$. 
Using the estimates for $\p^\ve(t+h, x) - \p^\ve(t, x)$ and $\n^\ve(t+h, x) - \n^\ve(t, x)$ 
together with the estimates for $\nabla \n^\ve$ and $\nabla \p^\ve$ in \eqref{apriori_estim} and  the properties of the extension  of $\n^\ve$ and $\p^\ve$ from $\Omega_M^\ve$ to $\Omega$, see Lemma~\ref{lem:extension},   and applying the Kolmogorov theorem \cite{Brezis,necas} we obtain the strong convergence of  $\n^\ve$  and $\p^\ve$  in $L^2(\Omega_T)$. 

 In the same way as in \cite{PS}  we show that, up to a subsequence,   
\begin{eqnarray*}
\cT^\ast_\ve(b^\ve)  \to b \quad   \text{ strongly}  \text{ in } L^2(\Omega_T\times Y_M),  \quad  \text{ as } \ve \to 0.
\end{eqnarray*}
  Here we present only the sketch of the calculations. 
Using the extension of $\n^\ve$ from $\Omega_M^\ve$ to $\Omega$, see  Lemma~\ref{lem:extension}, we define 
 the extension of $b^\ve$ from $\Omega_M^\ve$ to $\Omega$ as a solution of  the ordinary differential equation 
 \begin{equation}\label{extend_n_b}
 \begin{aligned}
& \partial_t b^\ve = R_b(\n^\ve, b^\ve, \mathcal N_\delta(\be(\bu^\ve)))
 \qquad  && \text{ in } (0,T)\times\Omega, \\
& b^\ve(0,x) = b_0  && \text{ in } \Omega.
 \end{aligned} 
 \end{equation}
 The construction of the extension for $\n^\ve$ and the uniform boundedness of $\n^\ve_1$, $\n^\ve_2$   in $\Omega_{M,T}^\ve$, see \eqref{apriori_estim}, ensure 
 $$
 \| \n^\ve \|_{L^\infty(0,T; L^\infty(\Omega))} \leq C  \|\n^\ve \|_{L^\infty(0,T; L^\infty(\Omega_{M}^\ve))}, 
 $$
 with the constant $C$ independent of $\ve$. Hence from  \eqref{extend_n_b} we obtain also the boundedness of $b^\ve$ and $\partial_t b^\ve$.   
 We show the strong convergence of $b^\ve$  by applying the Kolmogorov  theorem \cite{Brezis,necas}.  Considering equation  \eqref{extend_n_b} at $(t,x+\textbf{h}_j)$ and $(t,x)$, where   ${\bf h}_j = h \mathbf{b}_j$, with $(\mathbf{b}_1, \mathbf{b}_2, \mathbf{b}_3)$ being the canonical basis in $\mathbb R^3$ and $h>0$,  taking $b^\ve(t,x+{\bf h}_j) -  b^\ve(t,x)$ as a test function and using the Lipschitz continuity of $R_b$ yield 
 \begin{equation*} 
 \begin{aligned}
& \| b^\ve(\tau, \cdot+\textbf{h}_j) - b^\ve(\tau,\cdot) \|^2_{L^2(\Omega_{2h})} \leq 
 \| b_0(\cdot+{\mathbf h}_j) - b_0(\cdot) \|^2_{L^2(\Omega_{2h})} + C_1 \int_0^\tau  \|b^\ve(t, \cdot+{\bf h}_j) -  b(t, \cdot)\|^2_{L^2(\Omega_{2h})}  dt 
  \\ 
 &+ 
 C_2\int_0^\tau  \Big(  \|\n^\ve(t, \cdot+{\bf h}_j) -  \n(t, \cdot)\|^2_{L^2(\Omega_{2h})}  
  +   \delta^{-6}\Big \| \int_{B_{\delta,h}(x)\cap \Omega}{\rm tr }\, \bbE^\ve(b^\ve) \be(\bu^{\ve}(t,\tilde x)) d\tilde x \Big\|^2_{L^2(\Omega_{2h})} \Big) dt 
  \end{aligned}
 \end{equation*}
for  $\tau \in (0, T]$,  where
 $\Omega_{2h} = \{ x \in \Omega  \; | \; \text{dist} (x, \partial \Omega) \geq 2h \}$, 
 $B_{\delta, h}(x)= \big[B_\delta(x+\textbf{h}_j) \setminus B_\delta(x) \big]\cup \big[B_\delta(x) \setminus B_\delta(x+\textbf{h}_j) \big]$, and the constants $C_1, C_2$ are independent of $\ve$ and $h$.
 Using the regularity of the initial condition $b_{0} \in H^1(\Omega)$, the \textit{a priori} estimates for $\be(\bu^\ve)$ and 
 $\nabla \n^\ve$,  along with  the fact that
 $|B_{\delta, h}(x)\cap \Omega| \leq C \delta^2 h $ for all $x\in \overline \Omega$, and applying the Gronwall inequality we  obtain 
  \begin{eqnarray} \label{estim_strong1}
 \sup\limits_{t\in (0,T)}\| b^\ve(t, \cdot +\textbf{h}_j) - b^\ve(t,\cdot) \|^2_{L^2(\Omega_{2h})} \leq C_\delta h.
 \end{eqnarray} 
 Extending  $b^\ve$ by zero from $\Omega_T$ into $\mathbb R_{+}\times \mathbb R^3$ and using the uniform boundedness of $b^\ve$ in $L^\infty(0,T; L^\infty(\Omega))$ imply 
 \begin{eqnarray}\label{estim_strong2}
  \|b^\ve\|^2_{L^\infty(0,T; L^2(\tilde \Omega_{2h}))} +\|b^\ve\|^2_{L^2((T-h,T+h)\times  \Omega)} \leq C h, 
 \end{eqnarray}
where $\tilde \Omega_{2h} = \{ x\in \mathbb R^3 \; | \; \text{dist}(x, \partial \Omega)\leq 2h\}$ and the constant $C$ is independent of $\ve$ and $h$. The  estimates for  $\partial_t b^\ve$ ensure
 \begin{eqnarray}\label{estim_strong3}
\|b^\ve(\cdot+h, \cdot) - b^\ve(\cdot,\cdot) \|^2_{L^2((0,T-h)\times\Omega)} \leq C_1 h^2 \|\partial_t b^\ve \|^2_{L^2(\Omega_T)}\leq C_2 h^2, 
 \end{eqnarray}
where $C_1$ and $C_2$ are independent of $\ve$ and $h$. Combining \eqref{estim_strong1}--\eqref{estim_strong3} and applying  the Kolmogorov theorem yield the strong convergence of $b^\ve$ to $\widetilde b^\vartheta$ in $L^2(\Omega_T)$.   The definition of the two-scale convergence yields that  $\widetilde b^\vartheta= b^\vartheta$ and hence the two-scale limit of $b^\ve$ is independent of $y$.  Then using the properties of the unfolding operator, see e.g.\   \cite{CDG,CDDGZ}, we obtain the strong convergence of $\cT^\ast_\ve(b^\ve)$.

Considering an extension $\overline{\partial_t \bu^\ve}$ of $\partial_t \bu^\ve$ from $\Omega_M^\ve$ into $\Omega$  and applying the  Korn inequality \cite{OShY}  yield
\begin{align}\label{estim_visco_time_2}
\|\partial_t \bu^\ve\|_{L^2(0,T;H^1(\Omega_M^\ve))} &\leq 
\|\overline{\partial_t \bu^\ve}\|_{L^2(0,T;H^1(\Omega))} 
\leq C_1 \big[ \|\overline{\partial_t \bu^\ve}\|_{L^2(\Omega_T)}
+  \|\be(\overline{\partial_t \bu^\ve})\|_{L^2(\Omega_T)}\big] \nonumber \\
&\leq 
C_2 \big[ \|\partial_t \bu^\ve\|_{L^2(\Omega_{M,T}^\ve)}
+  \|\be(\partial_t\bu^\ve)\|_{L^2(\Omega_{M,T}^\ve)}\big] \leq C_3 \vartheta^{-\frac 12},
\end{align}
where the constant $C_3$ is independent of $\ve$ and $\vartheta$.\\
 Estimates \eqref{apriori_visco_2} and \eqref{estim_visco_time_2} ensure  the existence of  $\bu^\vartheta \in L^2(0,T; \cW(\Omega))$, 
$\bu^\vartheta_1 \in L^2(\Omega_T; H^1_{\text{per}}(\hat Y))$,   $\boldsymbol{\xi}^\vartheta\in L^2(0,T; H^1(\Omega))$ and  $\boldsymbol{\xi}^\vartheta_1\in L^2(\Omega_T; H^1_\text{per}(\hat Y_M))$  such that
\begin{equation*}
\begin{aligned}
\bu^\ve &  \rightharpoonup \bu^\vartheta, \; \quad && \phantom{\chi_{\hat Y_M} (\partial_t}\nabla \bu^\ve  \rightharpoonup  \nabla \bu^\vartheta + \hat \nabla_y \bu_1^\vartheta \quad \qquad\;  \text{ two-scale}, \\ 
\chi_{\Omega_M^\ve} \partial_t \bu^\ve &  \rightharpoonup \chi_{\hat Y_M} \boldsymbol{\xi}^\vartheta,  \quad && \chi_{\Omega_M^\ve} \nabla \partial_t \bu^\ve   \rightharpoonup \chi_{\hat Y_M} (\nabla  \boldsymbol{\xi}^\vartheta +\hat \nabla _y \boldsymbol{\xi}_1^\vartheta )   \quad \; \text{ two-scale}, 
\end{aligned}
\end{equation*}
see e.g.\  \cite{allaire}. Considering the  two-scale convergence of $\bu^\ve$ and $\partial_t \bu^\ve$,  we obtain 
\begin{eqnarray*}
\frac{|\hat Y_M|}{|\hat Y|}\langle \boldsymbol{\xi}^\vartheta, \phi  \rangle_{\Omega_T} = \lim\limits_{\ve \to 0}  \langle\partial_t\bu^\ve, \phi \rangle_{\Omega^\ve_{M,T}} = 
-  \lim\limits_{\ve \to 0} \langle \bu^\ve,  \partial_t\phi \rangle_{\Omega^\ve_{M,T}} 
 = -\frac{|\hat Y_M|}{|\hat Y|}\langle \bu^\vartheta,  \partial_t \phi \rangle_{\Omega_T} 
\end{eqnarray*}
for all $\phi \in C^\infty_0(\Omega_T)$. Hence,   $ \partial_t  \bu^\vartheta \in L^2(\Omega_T)$,  and $\boldsymbol{\xi}^\vartheta = \partial_t \bu^\vartheta$ a.e.\ in $\Omega_T\times\hat Y$.
The two-scale convergence of $\nabla \bu^\ve $ and $\partial_t\nabla \bu^\ve$ implies
\begin{eqnarray*}
&& |\hat Y|^{-1}\langle \partial_t \nabla  \bu^\vartheta  + \hat \nabla_y \boldsymbol{\xi}_1^\vartheta,  \phi  \rangle_{\Omega_T\times \hat Y_M} = \lim\limits_{\ve \to 0}  \langle \partial_t \nabla \bu^\ve,  \phi \rangle_{\Omega^\ve_{M,T}} \\ &&\qquad
 =
-  \lim\limits_{\ve \to 0} \langle \nabla \bu^\ve,  \partial_t\phi \rangle_{\Omega^\ve_{M,T}} 
 = - |\hat Y|^{-1}\langle \nabla \bu^\vartheta +\hat \nabla_y \bu_1^\vartheta ,  \partial_t \phi \rangle_{\Omega_T\times \hat Y_M} 
\end{eqnarray*}
for all  $\phi  \in C^\infty_0(\Omega_T; C^\infty_\text{per}(\hat Y))$. 
Thus, $\partial_t \hat \nabla_y \bu_1^\vartheta \in L^2(\Omega_T\times \hat Y_M)$ and  $\hat \nabla_y \boldsymbol{\xi}_1^\vartheta = \partial_t\hat  \nabla_y   \bu_1^\vartheta $ a.e.\ in $\Omega_T\times \hat Y_M$.
Therefore,   $\bu^\vartheta \in H^1(0,T; \cW(\Omega))$,  $\partial_t \bu_1^\vartheta\in L^2(\Omega_T; H^1_{\text{per}}(\hat Y_M))$ and 
$ \chi_{\Omega_M^\ve}\partial_t\be(\bu^\ve) \to \chi_{\hat Y_M}(\partial_t  \be(\bu^\vartheta) +\partial_t \hat \be_y(\bu_1^\vartheta))$   two-scale. 
\end{proof}


To derive macroscopic equations for the microscopic problem \eqref{visco2}--\eqref{BC},    we first derive the macroscopic equations for the perturbed system \eqref{reactions}, \eqref{BC},   \eqref{perturb_eq}. Then letting the perturbation parameter to go to zero we obtain the macroscopic equations for \eqref{visco2}--\eqref{BC}. 
 
\begin{theorem}\label{conver_visco_vartheta}
A sequence of solutions $( \bu^\ve, \p^\ve, \n^\ve, b^\ve)$,  of the microscopic problem \eqref{reactions}, \eqref{BC},   \eqref{perturb_eq},  converges to a solution  $( \bu^\vartheta, \p^\vartheta, \n^\vartheta, b^\vartheta)$ of the macroscopic perturbed equations 
 \begin{eqnarray}\label{macro_vis_nu}
\begin{aligned}
&\vartheta \bu^\vartheta_{tt} -{\rm div}\Big( \mathbb E^\vartheta_{\text{hom}}\be(\bu^\vartheta)  +  \mathbb V^\vartheta_{\text{hom}} \be(\bu^\vartheta_t) + \int_0^t   \mathbb K^\vartheta(t-s,s)  \be(\bu^\vartheta_s)  ds \Big)={\bf 0}  \; \text{ in }  \Omega_T, \\
&\big(\mathbb E^\vartheta_{\text{hom}}\be(\bu^\vartheta)  +  \mathbb V^\vartheta_{\text{hom}}\be(\bu^\vartheta_t) + \int_0^t  \mathbb K^\vartheta(t-s,s) \be(\bu^\vartheta_s) ds \big)\bnu = \bff_\cE \;   \quad \; \; \text{ on } \Gamma_{\cE,T},  \qquad \\
&\big(\mathbb E^\vartheta_{\text{hom}}\be(\bu^\vartheta)  +  \mathbb V^\vartheta_{\text{hom}} \be(\bu^\vartheta_t) + \int_0^t  \mathbb K^\vartheta(t-s,s)\be(\bu^\vartheta_s)  ds\big) \bnu = - p_\cI \bnu \;    \;  \text{ on } \Gamma_{\cI,T}, \\
&\bu^\vartheta(0,x) = \bu_0(x), \qquad \bu^\vartheta_t(0,x) ={\bf 0} \; \;    \text{ in } \Omega,
\end{aligned}
\end{eqnarray}
and 
\begin{equation}\label{macro_1_vartheta}
\begin{aligned}
  \partial_t  \p^\vartheta &= {\rm div} (\mathcal D_p \nabla \p^\vartheta) - \F_p (\p^\vartheta) ,   \\
\partial_t \n^\vartheta&= {\rm div} (\mathcal D_n \nabla \n^\vartheta) +  \F_n(\p^\vartheta, \n^\vartheta)+  \R_n(\n^\vartheta, b^\vartheta,  \mathcal N_\delta^{\rm eff}(\be(\bu^\vartheta))) ,  \\
\partial_t b^\vartheta &=  \phantom{{\rm div} (\mathcal D_c \nabla n_c)-}  \; \;  R_b(\n^\vartheta, b^\vartheta,  \mathcal N_\delta^{\rm eff} (\be(\bu^\vartheta)))   
\end{aligned}
\end{equation}
in $\Omega_T$
   together with the initial and  boundary conditions 
\begin{equation}\label{macro_bc_vartheta}
\begin{aligned}
& \mathcal D_p \nabla \p^\vartheta \,  \bnu = \theta_M^{-1} \J_p (\p^\vartheta),  &&  
\mathcal D_n \nabla \n^\vartheta \,  \bnu =\theta_M^{-1} \G(\n^\vartheta) \mathcal N_\delta^{\rm eff}(\be(\bu^\vartheta)) & & \text{on } \Gamma_{\cI,T}, \\
&\mathcal D_p \nabla \p \, \bnu = - \theta_M^{-1}\gamma_p  \, \p^\vartheta,    && \mathcal D_n \nabla \n^\vartheta \, \bnu =  \theta_M^{-1} \J_n (\n^\vartheta)   &&  \text{on } \Gamma_{\cE,T},  \\
&\mathcal D_p \nabla \p^\vartheta \,  \bnu = 0,  \quad  &&  \mathcal D_n \nabla \n^\vartheta \,  \bnu = 0 && \text{on } \Gamma_{\cU,T}, \\
&\p^\vartheta, \quad \n^\vartheta && a_3-\text{periodic in } x_3, \\
& \p^\vartheta (0) = \p_0 (x), \quad \n^\vartheta(0)= \n_0, \quad && b(0)= b_0 &&\text{in }  \Omega,  
\end{aligned}
\end{equation}
where  $ \mathbb E^\vartheta_{\text{hom}}=  \mathbb E^\vartheta_{\text{hom}}(t,b^\vartheta)$, $ \mathbb V^\vartheta_{\text{hom}}=  \mathbb V^\vartheta_{\text{hom}}(t,b^\vartheta)$, and  $\mathbb  K^\vartheta(t,s)= \mathbb  K^\vartheta(t,s, b^\vartheta)$ are defined by 
\begin{eqnarray}\label{macro_coef_2}
\begin{aligned}
 \mathbb E^{\vartheta}_{\text{hom},ijkl}(b^\vartheta) &= \dashint_{\hat Y} \big[  \mathbb E_{ijkl}(b^\vartheta,y) + \big(\mathbb E(b^\vartheta,y) \hat \be_y(\mathbf{w}_\vartheta^{ij})\big)_{kl}  +\big(\mathbb V_M(b^\vartheta,y) \partial_t\hat \be_y(\mathbf{w}_\vartheta^{ij})\big)_{kl}  \chi_{\hat Y_M} \big] dy, \\
  \mathbb V^{\vartheta}_{\text{hom},ijkl}(b^\vartheta) &= \frac 1{|\hat Y|}\int_{\hat Y_M} \big[  \mathbb V_{M, ijkl}(b^\vartheta,y) + \big(\mathbb V_M(b^\vartheta,y) \hat \be_y(\boldsymbol{\chi}^{ij}_{\mathbb V, \vartheta})\big)_{kl} \big] dy,\\
   \mathbb K^{\vartheta}_{ijkl}(t,s,  b^\vartheta) &= \dashint_{\hat Y} \big[  \big(\mathbb E(b^\vartheta(t+s),y) \hat \be_y(\mathbf{v}^{ij}_\vartheta(t,s))\big)_{kl} + \big(\mathbb V_M(b^\vartheta(t+s),y) \partial_t\hat \be_y(\mathbf{v}_\vartheta^{ij}(t,s))\big)_{kl} \chi_{\hat Y_M}\big] dy, 
  \end{aligned}
\end{eqnarray}
and 
  $\mathbf{w}^{ij}_\vartheta(t,x,y)$, $\boldsymbol{\chi}^{ij}_{\mathbb V, \vartheta}(t,x,y)$, and  $\mathbf{v}^{ij}_\vartheta$ are solutions of the unit cell problems  
  \eqref{unit_chi} and \eqref{unit_wv} with $b^\vartheta$ instead of $b$.
The macroscopic diffusion matrices $\mathcal D_\alpha^l$, with $\alpha=n,p$ and $l=1,2$, are defined as in \eqref{macro_diff_coef} and  $\mathcal N_\delta^{\rm eff}$ is defined in \eqref{macro_N_d_visco}.
\end{theorem}

\begin{proof}
To pass to the limit in the equations for $\n^\ve$ and $b^\ve$,  we shall prove the strong convergence of $\int_\Omega \be(\bu^\ve) dx$ in $L^2(0,T)$ using the Kolmogorov compactness theorem \cite{Brezis,necas}.  
Considering the difference of  \eqref{perturb_eq} for $t$ and $t+h$  and taking $\delta^h \bu^\ve(t,x)= \bu^\ve(t+h, x) - \bu^\ve(t,x)$ as a test function yield 
\begin{eqnarray}\label{estim_diff_11}
\begin{aligned}
&\int_0^{T-h} \Big[\langle \mathbb E^\ve(b^\ve(t+h))\be(\bu^\ve(t+h)) -\mathbb E^\ve(b^\ve(t))\be(\bu^\ve(t)),
 \be(\delta^h\bu^\ve) \rangle_{\Omega}
\\&\quad + \langle \mathbb V^\ve(b^\ve(t+h))\partial_t\be(\bu^\ve(t+h))- \mathbb V^\ve(b^\ve(t))\partial_t\be(\bu^\ve(t)), \be(\delta^h\bu^\ve)  \rangle_{\Omega_M^\ve}\Big] dt \\
&\quad  +
 \vartheta \langle  \delta^h\partial_t\bu^\ve(T-h), \delta^h\bu^\ve(T-h) \rangle_{\Omega^\ve_{M}}
  - \vartheta \langle \delta^h\partial_t \bu^\ve(0), \delta^h \bu^\ve(0) \rangle_{\Omega^\ve_{M}}\\&\qquad=
 \int_0^{T-h}  \Big[\vartheta  \|\delta^h \partial_t \bu^\ve\|^2_{L^2(\Omega_M^\ve)}+
  \langle \delta^h \bff_\cE, \delta^h \bu^\ve \rangle_{L^2(\Gamma_\cE)} - 
  \langle \delta^h p_\cI \bnu, \delta^h \bu^\ve \rangle_{L^2(\Gamma_\cE)}\Big] dt .
  \end{aligned}
\end{eqnarray}
To estimate the first term on the right-hand side we consider  $\delta^h \bu^\ve_t(t,x)= \bu^\ve_t (t+h,x) -  \bu^\ve_t (t,x)= \int_{t}^{t+h} \partial^2_{\tau} \bu^\ve(\tau, x)  d\tau$,  integrate  \eqref{perturb_eq}  
over $(t, t+h)$ and take $\overline {\bu^\ve_t}(t+h,x) - \overline {\bu^\ve_t}(t,x)$ as a test function, with $\overline u_t^\ve$ being an extension of $u_t^\ve$ from $\Omega_M^\ve$ to $\Omega$ as in Lemma~\ref{lem:extension},
\begin{eqnarray}\label{estim_deriv_h}
\begin{aligned}
&\vartheta  \|\delta^h\bu^\ve_t\|^2_{L^2((0, T-h)\times\Omega_M^\ve)}
 \leq   h C_1 \big[\|p_{\mathcal I}\|_{H^1(0,T; L^2(\Gamma_{\mathcal I}))} + \|\bff_{\mathcal E}\|_{H^1(0,T; L^2(\Gamma_{\mathcal E}))} \big]\\
 &+h^{\frac 12}  C_2 \big[\|\be(\bu^\ve)\|^2_{L^\infty(0,T; L^2(\Omega))}   + \|\be(\overline {\bu^\ve_t})\|^2_{L^2(\Omega_T)}   + \|\be(\bu^\ve_t)\|^2_{L^2(\Omega_{M,T}^\ve)} \big]\leq C h^{\frac 12},\quad
\end{aligned}
\end{eqnarray}
where the constant $C$ is independent of $\ve$, $\vartheta$, and $h\in (0,T)$.  Here we used  estimates \eqref{apriori_visco_2} and  the  property of the extension, i.e.\
$ \|\be(\overline {\bu^\ve_t})\|^2_{L^2(\Omega_T)} \leq C_1 \|\be(\bu^\ve_t)\|^2_{L^2(\Omega_{M,T}^\ve)}$ with a constant $C_1$ independent of $\ve$, see e.g.\  \cite{OShY}.

Using the estimate for $\vartheta^{1/2}\|\partial_t \bu^\ve\|_{L^\infty(0,T; L^2(\Omega_M^\ve))}$ in \eqref{apriori_visco_2} we obtain 
\begin{equation}\label{estim_111}
\begin{aligned}
\vartheta \langle  \delta^h \partial_t \bu^\ve(T-h), \delta^h\bu^\ve(T-h) \rangle_{\Omega^\ve_{M}}
 \leq 2 \vartheta \|\partial_t \bu^\ve\|_{L^\infty(0, T; L^2(\Omega^\ve_{M}))}\|\delta^h\bu^\ve(T-h)\|_{L^2(\Omega^\ve_{M})}\\
\leq  C \vartheta^{1/2}  \big\|\int_{T-h}^T \partial_t \bu^\ve dt \big\|_ {L^2(\Omega^\ve_{M})} \leq C h \, \vartheta^{1/2} \|\partial_t \bu^\ve  \|_ {L^\infty(0,T; L^2(\Omega^\ve_{M}))} \leq C h.
\end{aligned}
\end{equation}
In the same way  we also  have 
\begin{equation}\label{estim_112}
\vartheta \langle \delta^h\partial_t \bu^\ve(0), \delta^h \bu^\ve(0) \rangle_{\Omega^\ve_{M}} \leq Ch,
\end{equation}
where  $C$ is independent of $\ve$, $\vartheta$, and $h$.  To estimate  the first two terms on the left-hand side of \eqref{estim_diff_11} we use  the uniform boundedness of $b^\ve$ and $\partial_t b^\ve$,  the equality $\delta^h\be(\bu^\ve(t,x)) = h \int_0^1 \partial_t \be(\bu^\ve(t + hs, x)) ds$,  and estimates \eqref{apriori_visco_2}:
\begin{eqnarray}\label{estim_113}
\begin{aligned}
\int_0^{T-h} \langle (\mathbb E^\ve(b^\ve(t+h)) -\mathbb E^\ve(b^\ve(t)))\be(\bu^\ve(t)),
 \be(\delta^h\bu^\ve(t)) \rangle_{\Omega} dt 
 \leq h C_1  \|\partial_t b^\ve \|_{L^\infty(\Omega_{M,T}^\ve)}  \|\be(\bu^\ve) \|^2_{L^2(\Omega_{T})} \leq C_2h,\\
\int_0^{T-h}\langle \mathbb V^\ve(b^\ve(t+h),x)\partial_t\be(\bu^\ve(t+h)) - \mathbb V^\ve(b^\ve(t),x)\partial_t\be(\bu^\ve(t)), \delta^h\be(\bu^\ve(t))  \rangle_{\Omega_M^\ve} dt  \hspace{ 3.6 cm } \\
\leq h C_3\| b^\ve \|_{L^\infty(\Omega_{M,T}^\ve)}  \|\partial_t\be(\bu^\ve) \|^2_{L^2(\Omega_{M,T}^\ve)}\leq C_4 h,
 \end{aligned}
\end{eqnarray}
with the constants $C_j$, $j=1,2,3,4$,  independent of $\ve$, $\vartheta$, and $h$.
Then,  the assumptions on $\mathbb E$, $\bff_\cE$, and $p_\cI$,   estimates \eqref{apriori_visco_2}  and \eqref{estim_deriv_h}--\eqref{estim_113}, and the boundedness of $b^\vartheta$ ensures 
\begin{eqnarray}\label{Kolmog}
\begin{aligned}
& \| \be(\bu^\ve(t+h)) - \be(\bu^\ve(t)) \|^2_{L^2((0,T-h)\times\Omega)}\leq  C h^{1/2}, \\
&\| \be(\bu^\ve) \|^2_{L^2((T-h,T)\times\Omega)}\leq  h \| \be(\bu^\ve) \|^2_{L^\infty(0,T; L^2(\Omega))} \leq C h,
\end{aligned}
\end{eqnarray}
with a constant $C$ independent of $\ve$, $\vartheta$, and $h$.

Thus,  the estimate \eqref{Kolmog}, along with the estimate for $\partial_t b^\ve$,  the Kolmogorov theorem,  and the two-scale convergence of $\bu^\ve$, yields the strong convergence, 
 up to a subsequence, 
 $$
 \begin{aligned}
 \int_\Omega \be(\bu^\ve) dx & \to \int_\Omega \dashint_{\hat Y} [\be(\bu^\vartheta)+ \hat \be_y(\hat\bu^\vartheta)] dy dx \quad  \text{ in } L^2(0,T), \\
 \int_\Omega \mathbb E(b^\vartheta, x/\ve)\be(\bu^\ve) dx &\to \int_\Omega \dashint_{\hat Y} \mathbb E(b^\vartheta, y)(\be(\bu^\vartheta)+ \hat \be_y(\hat\bu^\vartheta)) dy dx\quad  \text{ in } L^2(0,T), \quad \text{ as } \; \ve \to 0.
 \end{aligned}
 $$

Now we can pass to the limit as $\ve\to 0$ in the microscopic equations  \eqref{reactions}, \eqref{BC}, and \eqref{perturb_eq},  with initial and boundary conditions in  \eqref{visco2} and \eqref{in_perturbed}.
Considering  $\bphi_\alpha(t,x) =\boldsymbol{\varphi}_\alpha(t,x)  +\ve \boldsymbol{\psi}_\alpha(t,x, \hat x/\ve)$ as a test function  in \eqref{weak_sol_n1}--\eqref{weak_sol_n2}, 
 where $\boldsymbol{\varphi}_\alpha \in C^\infty(\overline \Omega_T)$ and   $a_3$-periodic in $x_3$,  and $\boldsymbol{\psi}_\alpha \in C^\infty_0(\Omega_T; C^\infty_{\text{per}}(\hat Y))$, for $\alpha=1,2$,  applying the two-scale convergence and using  the strong convergence of  $\cT^\ast_\ve(b^\ve)$ and $\p^\ve$, $\n^\ve$,  see  Lemma~\ref{converg_vartheta},  along with 
  strong convergence of $\int_\Omega \be(\bu^\ve) dx$ and  $\int_\Omega \mathbb E(b^\vartheta, x/\ve)\be(\bu^\ve) dx$, 
 we obtain macroscopic equations  \eqref{macro_1_vartheta}--\eqref{macro_bc_vartheta} for $\p^\vartheta$, $\n^\vartheta$, and $b^\vartheta$ in the same way as in \cite{PS}. 

 Using the strong convergence of $\cT^\ast_\ve(b^\ve)$ along with  the two-scale convergence of  $\bu^\ve$, $\be(\bu^\ve)$ and $\partial_t \be(\bu^\ve)$, as $\ve \to 0$, yields  the macroscopic equations
\begin{eqnarray}\label{macro_3}
\begin{aligned}
 \langle \bbE(b^\vartheta,y) (\be(\bu^\vartheta) +\hat  \be_y(\hat \bu^\vartheta)) + 
 \mathbb V(b^\vartheta,y) \partial_t (\be(\bu^\vartheta) + \hat \be_y(\hat \bu^\vartheta)), \be(\boldsymbol{\psi})  + \hat \be_y(\boldsymbol{\psi}_1) \rangle_{\Omega_T\times \hat Y} \qquad 
\\- \vartheta|\hat Y|\langle \partial_t \bu^\vartheta, \partial_t \boldsymbol{\psi} \rangle_{\Omega_T} = |\hat Y|  \big[\langle \mathbf{f}, \boldsymbol{\psi} \rangle_{ \Gamma_{\cE,T}} - 
\langle p_{\cI}\bnu, \boldsymbol{\psi} \rangle_{ \Gamma_{\cI,T}} \big]\qquad 
\end{aligned}
\end{eqnarray}
for $\boldsymbol{\psi} \in C^1_0(0,T; C^1(\overline \Omega))^3$, with $\boldsymbol{\psi}$  being $a_3$-periodic in $x_3$, and $\boldsymbol{\psi}_1 \in C^\infty_0(\Omega_T; C^\infty_\text{per}(\hat Y))^3$.

 Taking $\boldsymbol{\psi}\equiv \textbf{0}$ we obtain 
\begin{eqnarray}\label{macro_4}
\langle \bbE(b^\vartheta,y) (\be(\bu^\vartheta) + \hat \be_y(\hat \bu^\vartheta))  + 
 \mathbb V(b^\vartheta,y) \partial_t  (\be(\bu^\vartheta) + \hat \be_y(\hat \bu^\vartheta)),  \hat \be_y(\boldsymbol{\psi}_1) \rangle_{\Omega_T\times \hat Y}=0.\qquad 
\end{eqnarray}
Considering  the structure of  \eqref{macro_4} and taking into account the fact that $\bbE(b^\vartheta,\cdot)$ and $\bbV(b^\vartheta, \cdot)$ depend on $t$, we seek $\hat \bu^\vartheta$ in the form 
$$
\hat \bu^\vartheta (t,x,y)= \sum_{i,j=1}^{3} \Big[\be(\bu^\vartheta(t,x))_{ij} \mathbf{w}^{ij}_\vartheta(t, x, y) + \int_0^t  \partial_s \be(\bu^\vartheta(s,x))_{ij} \mathbf{v}^{ij}_\vartheta(t-s,s,x, y) ds \Big]
$$
 and rewrite the equation  \eqref{macro_4}  as 
\begin{eqnarray}\label{macro_5}
\begin{aligned}
&\Big\langle \bbE(b^\vartheta, y) \Big(\be(\bu^\vartheta) + \sum_{i,j=1}^{3}
 \Big[ \be(\bu^\vartheta)_{ij} \hat \be_y(\mathbf{w}^{ij}_\vartheta) +  \int_0^t  \partial_s \be(\bu^\vartheta)_{ij} \hat \be_y(\mathbf{v}^{ij}_\vartheta)  ds \Big] 
\Big ),\hat \be_y(\boldsymbol{\psi}_1)\Big \rangle_{\Omega_T\times \hat Y} \\
&\qquad+
\Big\langle \mathbb V_M(b^\vartheta,y) \Big(\partial_t \be(\bu^\vartheta) 
 +\sum_{i,j=1}^{3} \Big[  \partial_t \be(\bu^\vartheta)_{ij} \hat \be_y(\mathbf{w}^{ij}_\vartheta) + \be(\bu^\vartheta)_{ij} \partial_t \hat \be_y(\mathbf{w}^{ij}_\vartheta)  \\
&  \qquad+ \partial_t \be(\bu^\vartheta)_{ij}\hat  \be_y(\mathbf{v}^{ij}_\vartheta(0,t,x,y)) +  
\int_0^t \partial_s \be(\bu^\vartheta)_{ij} \partial_t\hat  \be_y(\mathbf{v}^{ij}_\vartheta)  ds \Big]  \Big),  \hat \be_y(\boldsymbol{\psi}_1)\Big \rangle_{\Omega_T\times \hat Y_M}=0.
\end{aligned}
\end{eqnarray}
Considering the terms with  $\be(\bu^\vartheta)$ and $\partial_t \be(\bu^\vartheta)$, respectively,  we obtain that 
$\mathbf{v}^{ij}_\vartheta(0,t,x,y) = \boldsymbol{\chi}^{ij}_{\mathbb V, \vartheta}(t,x,y) -\mathbf{w}^{ij}_\vartheta(t,x,y) $ a.e.\ in $\Omega_T\times \hat Y_M$, 
where  $\mathbf{w}^{ij}_\vartheta(t,x,y)$ and $\boldsymbol{\chi}^{ij}_{\mathbb V, \vartheta}(t,x,y)$ are solutions of the unit cell problems   
\eqref{unit_chi} with $b^\vartheta$ instead of $b$. Using this  in \eqref{macro_5}  implies   that  $\mathbf{v}^{ij}_\vartheta$  satisfies  \eqref{unit_wv} with $b^\vartheta$ instead of $b$.  
Taking  $\boldsymbol{\psi}_1 \equiv \textbf{0}$ in \eqref{macro_3} yields  the macroscopic    equations  \eqref{macro_vis_nu} for $\bu^\vartheta$.

In the same way as for  the macroscopic elasticity tensor  for the equations of linear elasticity, see e.g.\  \cite{JKO, OShY}, we obtain that $\mathbb V^\vartheta_{\rm hom}$ is positive-definite and possesses major and minor symmetries, as in Assumption~\ref{assumptions}.8. 
The assumptions on $\mathbb E$ and $\mathbb V_M$ and the uniform boundedness of $b^\vartheta$ ensure the boundedness of 
$\mathbb E^\vartheta_{\rm hom}$ and $\mathbb K^\vartheta$.  Notice that the positive-definiteness and symmetry properties of $\mathbb V^\vartheta_{\rm hom}$ together with  the boundedness of  $\mathbb E^\vartheta_{\rm hom}$ and $\mathbb K^\vartheta$ ensure the well-possedness of the   viscoelastic equations \eqref{macro_vis_nu}.
\end{proof}

 Now we can complete the proof of the main result of the  paper. 
 
 \begin{proof}[Proof of Theorem~\ref{viscoelastic}]
To complete the proof  of Theorem~\ref{viscoelastic}, we  have to show that $\{\p^\vartheta\}$, $\{\n^\vartheta\}$, $\{b^\vartheta\}$, and $\{\bu^\vartheta\}$ converge to solutions of the 
macroscopic model \eqref{macro_1}--\eqref{macro_N_d_visco}. 
Using the fact that the estimates \eqref{apriori_visco_2} and \eqref{Kolmog} for $\bu^\ve$ are independent of $\vartheta$ and $\ve$ and applying the weak and two-scale convergence of $\bu^\ve$ together with the lower semicontinuity  of a norm  yield 
\begin{eqnarray}\label{estim_u_vartheta1}
\begin{aligned}
\|\bu^\vartheta\|^2_{L^\infty(0,T; \cW(\Omega))} + \|\be(\bu^\vartheta)+\hat  \be_y(\hat \bu^\vartheta)\|^2_{L^\infty(0,T; L^2(\Omega \times\hat Y))}  & \leq C, \\
   \| \be(\bu^\vartheta(\cdot+h, \cdot)) - \be(\bu^\vartheta) \|^2_{L^2((0,T-h)\times\Omega)} &
\leq  C h^{1/2}, 
\end{aligned}
\end{eqnarray}
with a constant $C$ independent of $\vartheta$ and $h$.

Similar to  the proof of  Lemma~\ref{th:exist_1},  using the estimates \eqref{estim_u_vartheta1} we obtain  the estimates for
$\p^\vartheta$ and 
 $\n^\vartheta$ in $L^2(0,T; \mathcal V(\Omega))\cap L^\infty(0,T; L^\infty(\Omega))$,  
  and  $b^\vartheta$ in $W^{1, \infty}(0,T; L^\infty(\Omega))$ uniformly in $\vartheta$. In a similar way as in the proof of Lemma~\ref{converg_vartheta}, 
we show  
\begin{equation}\label{estim_nb_Kolmog}
\begin{aligned}
&\|b^\vartheta(\cdot,\cdot+\textbf{h}_j)- b^\vartheta\|^2_{L^\infty(0,T; L^2(\Omega))} \leq C h, \\
&\|b^\vartheta(\cdot+h,\cdot)- b^\vartheta\|^2_{L^2(\Omega_T)} \leq C h, 
\end{aligned}
\end{equation}
where $b^\vartheta$ is extended by zero from $\Omega_T$ into $\mathbb R^3 \times \mathbb R_+$ and $\textbf{h}_j= h \mathbf{b}_j$, with $h\in (0, T)$. 
Then, applying the Kolmogorov theorem we obtain the strong convergence of  a subsequence of  $b^\vartheta$ in $L^2(\Omega_T)$ as $\vartheta \to 0$.

In a similar way as in the proof of  Lemma~\ref{th_2}, considering the assumptions on $\mathbb E$ and $\mathbb V$,  together with the boundedness of $b^\vartheta$ and $\partial_t b^\vartheta$, uniformly in $\vartheta$, we obtain 
the existence of weak solutions of the unit cell problems \eqref{unit_chi}, with $b^\vartheta$ instead of $b$, satisfying 
 \begin{eqnarray}\label{estim_w_chi}
 \begin{aligned}
& \| {\bf w}_\vartheta^{ij} \|^2_{L^\infty(0,T; H^1_{\text{per}}(\hat Y))} + \|\partial_t \hat \be_y({\bf w}^{ij}_\vartheta)\|^2_{L^2(0,T;  L^2(\hat Y_M))}  \leq C && \text{for a.a. } x\in \Omega , \\
&  \|\boldsymbol{\chi}^{ij}_{\mathbb V,\vartheta}\|^2_{H^1_\text{per} (\hat Y_M)}  \leq C  && \text{for a.a. } (t,x)\in \Omega_T, \quad
  \end{aligned}
  \end{eqnarray}
where the constant $C$ is independent of $\vartheta$.  The estimates \eqref{estim_w_chi} and boundedness of $b^\vartheta$ and $\partial_t b^\vartheta$  ensure  the existence of a weak solution of the unit cell problem   \eqref{unit_wv} with $b^\vartheta$ instead of $b$  such that 
  \begin{eqnarray}\label{estim_v_theta} 
  \| {\bf v}_\vartheta^{ij} \|^2_{L^\infty(0,T-s; H^1_{\text{per}}(\hat Y))} +  \|\partial_t\hat \be_y({\bf v}^{ij}_\vartheta)\|^2_{L^2(0,T-s;  L^2(\hat Y_M))} \leq C 
 \end{eqnarray}
for a.a.\ $x\in \Omega$ and  $s \in [0,T]$.

 Using the assumptions on $\mathbb V_M$, we obtain the symmetry properties  and strong ellipticity of $\mathbb V^\vartheta_{\text{hom}}$, see e.g.~\cite{SP,OShY}, with an ellipticity constant  independent of $\vartheta$.   The assumptions on $\mathbb E$ and $\mathbb V_M$, the uniform boundedness of $b^\vartheta$, and the estimates \eqref{estim_w_chi}--\eqref{estim_v_theta} ensure 
 \begin{equation}\label{estim_hom_tensor}
 \begin{aligned}
 \|\mathbb E^\vartheta_\text{hom}(b^\vartheta)\|_{L^2(0,T; L^\infty(\Omega))} + 
 \|\mathbb V^\vartheta_\text{hom}(b^\vartheta)\|_{L^\infty(0,T; L^\infty(\Omega))} +
 \|\mathbb K^\vartheta(t-s,s, b^\vartheta) \|_{L^2(0,T; L^\infty(0,t; L^\infty(\Omega)))} \leq C,
 \end{aligned}
 \end{equation}
with a constant $C$ independent of $\vartheta$. 

Taking $\bu^\vartheta_t$ as a test function in the weak formulation of \eqref{macro_vis_nu},  using   the strong ellipticity of $\mathbb V^\vartheta_{\text{hom}}$ together with 
estimates \eqref{estim_u_vartheta1} and \eqref{estim_hom_tensor},   and applying the second Korn inequality  for $\bu^\vartheta(t) \in \mathcal W(\Omega)$ 
  yield   
 \begin{equation}\label{apriori_u_theta}
 \vartheta \|\partial_t \bu^\vartheta\|^2_{L^\infty(0, T; L^2(\Omega))} +  \|\bu^\vartheta\|^2_{H^1(0, T; \mathcal W(\Omega))} \leq C, 
 \end{equation}
with a constant $C$ independent of $\vartheta$.   Hence we have  the weak convergence, up to a subsequence, of $\bu^\vartheta$ in $H^1(0,T; \cW(\Omega))$
and weak-$\ast$ convergence of $\vartheta^{1/2}  \partial_t \bu^\vartheta$ in $L^\infty(0, T; L^2(\Omega))$.
 
Hence,  to pass to the limit as $\vartheta \to 0$ in the macroscopic equations \eqref{macro_vis_nu} we have to show the strong convergence of $\mathbb E^\vartheta_\text{hom}$, $\mathbb V^\vartheta_\text{hom}$, and $\mathbb K^\vartheta$ as $\vartheta \to 0$.
  First, we show the strong convergence of $\int_{\hat Y} \hat \be_y({\bf w}^{ij}_\vartheta) dy$ and $\int_{\hat Y_M} \partial_t \hat \be_y({\bf w}^{ij}_\vartheta) dy$  in $L^2(\Omega_T)$.
Considering the first equation in  \eqref{unit_chi}, with $b^\vartheta$ instead of $b$, for $t+h$ and $t$, with  $h>0$, taking $\delta^h {\bf w}^{ij}_\vartheta(t,x,y)= {\bf w}^{ij}_\vartheta(t+h, x,y) - {\bf w}^{ij}_\vartheta(t, x,y)$ as a test function, and using $\delta^h \hat \be_y(\mathbf{w}^{ij}_\vartheta (t))= h\int_0^1\partial_t \hat \be_y(\mathbf{w}^{ij}_\vartheta)(t+h\tau) d\tau$, we obtain  
\begin{eqnarray}\label{estim5_1}
\begin{aligned}
\| \delta^h \hat \be_y(\mathbf{w}^{ij}_\vartheta )\|^2_{L^2((0,T-h)\times \hat Y)}
 \leq C_1 h \big[\|b^\vartheta\|_{L^\infty(0,T; L^\infty(\Omega))}\|\partial_t\hat \be_y(\mathbf{w}^{ij}_\vartheta)\|^2_{L^2(\hat Y_{M,T})} \qquad\\
+\|\partial_t b^\vartheta\|_{L^2(0,T; L^\infty(\Omega))}\|\hat \be_y(\mathbf{w}^{ij}_\vartheta)\|^2_{L^\infty(0,T; L^2(\hat Y))}\big]  \leq  C_2 h \qquad
\end{aligned}
\end{eqnarray}
for a.a.\ $x \in \Omega$ and the constants $C_1$ and $C_2$ are independent of $\vartheta$. 
Taking an extension $\overline{\delta^h\partial_t {\bf w}^{ij}_\vartheta}$  of $\delta^h \partial_t {\bf w}^{ij}_\vartheta$  from $\hat Y_M$ to $\hat Y$ as a test function in the weak formulation of \eqref{unit_chi}$_1$, with $b^\vartheta$ instead of $b$,  yields
\begin{eqnarray}\label{estim5_2}
\begin{aligned}
\|  \delta^h \hat \be_y(\partial_t\mathbf{w}^{ij}_\vartheta )\|^2_{L^2((0,T-h)\times \hat Y_M)}\leq C_1 
\|b^\vartheta\|^2_{L^\infty(0,T; L^\infty(\Omega))}\| \delta^h\hat  \be_y(\mathbf{w}^{ij}_\vartheta)\|^2_{L^2(\hat Y_{T-h})}\hspace{1.75in} \\
 + C_2 \big[1+ \| \hat \be_y(\mathbf{w}^{ij}_\vartheta)\|^2_{L^2(\hat Y_T) } +\|\hat \be_y(\partial_t \mathbf{w}^{ij}_\vartheta)\|^2_{L^2(\hat Y_{M,T}) } \big]\| \delta^h b^\vartheta\|^2_{L^\infty(0, T-h; L^\infty(\Omega))} 
\leq h C_3  \|b^\vartheta\|^2_{W^{1, \infty}(0,T; L^\infty(\Omega))}
\end{aligned}
\end{eqnarray}
for a.a.\ $x \in \Omega$ and the constants $C_1$, $C_2$, and $C_3$ are  independent of $\vartheta$ and $h$.  Here, we used the fact that 
due to  the periodicity of  ${\bf w}^{ij}_\vartheta $ and the Korn inequality  we have 
$$\|\delta^h \partial_t {\bf w}^{ij}_\vartheta \|_{L^2(0,T-h; H^1(\hat Y_M))} \leq C \|\delta^h \hat \be_y(\partial_t {\bf w}^{ij}_\vartheta)\|_{L^2((0,T-h)\times \hat Y_{M})}, $$
for a.a.\ $x \in \Omega$, and $\|\hat \be_y(\overline{\delta^h\partial_t {\bf w}^{ij}_\vartheta})\|_{L^2((0,T-h)\times \hat Y)} \leq C \|\hat \be_y(\delta^h\partial_t {\bf w}^{ij}_\vartheta)\|_{L^2((0,T-h)\times \hat Y_{M})}$, where the constant $C$  is independent of $\vartheta$. 

Considering    \eqref{unit_chi}$_1$, with $b^\vartheta$ instead of $b$, for $x+\textbf{h}_j$ and $x$, where $\textbf{h}_j = h\mathbf{b}_j$, and   using  \eqref{estim_nb_Kolmog} imply 
\begin{eqnarray}\label{estim5_3}
\| \delta^{\textbf{h}_j} \hat \be_y(\mathbf{w}^{ij}_\vartheta)\|^2_{L^\infty(0,T; L^2(\Omega\times \hat Y))} 
+ \|  \delta^{\textbf{h}_j}\hat \be_y(\partial_t\mathbf{w}^{ij}_\vartheta )\|^2_{L^2(\Omega_T\times \hat  Y_M)} 
\leq Ch, 
\end{eqnarray}
where $\delta^{\textbf{h}_j} \mathbf{w}^{ij}_\vartheta(t,x,y) = \mathbf{w}^{ij}_\vartheta(t, x+\textbf{h}_j, y) -  \mathbf{w}^{ij}_\vartheta(t, x, y)$,  the function  $b^\vartheta$ is extended by zero from $\Omega_T$ in $\mathbb R_+\times\mathbb R^3$, and  
$C$ is independent of $\vartheta$.
In the same manner we obtain 
\begin{eqnarray}\label{estim5_4}
\begin{aligned}
&\| \delta^h \hat \be_y(\boldsymbol{\chi}^{ij}_{\mathbb V, \vartheta})\|^2_{L^2(\Omega_T\times \hat  Y_M)}+
 \| \delta^{\textbf{h}_j} \hat \be_y(\boldsymbol{\chi}^{ij}_{\mathbb V, \vartheta})\|^2_{L^2(\Omega_T\times \hat Y_M)}\leq Ch, 
\end{aligned}
\end{eqnarray}
where $b^\vartheta$ and $\boldsymbol{\chi}^{ij}_{\mathbb V, \vartheta}$ are extended by zero from $\Omega_T$   into $\mathbb R_+\times\mathbb R^3$, and 
\begin{eqnarray}\label{estim5_5}
\begin{aligned}
&\|  \hat \be_y(\mathbf{v}^{ij}_\vartheta (t-s+h,s))- \hat \be_y(\mathbf{v}^{ij}_\vartheta (t-s,s))\|^2_{L^2(0,T-h; L^2(\Omega_t\times \hat Y))}\leq Ch, \\
&\| \delta^{\textbf{h}_j}\hat  \be_y(\mathbf{v}^{ij}_\vartheta(t-s,s))\|^2_{L^2(0, T; L^2(\Omega_t\times \hat Y))}  \leq Ch.  \\
\end{aligned}
\end{eqnarray}

 Considering the difference of  equations in   \eqref{unit_wv}, with $b^\vartheta$ instead of $b$,  for $s+h$ and $s$, taking $\hat \be_y(\mathbf{v}^{ij}_\vartheta (t, s+ h,x))- \hat \be_y(\mathbf{v}^{ij}_\vartheta (t,s,x))$ as a test function, and using the estimates for 
$\delta^h \hat \be_y(\mathbf{w}^{ij}_\vartheta)$ and $\delta^h\hat  \be_y(\boldsymbol{\chi}^{ij}_{\mathbb V, \vartheta})$  in \eqref{estim5_1} and \eqref{estim5_4}, respectively, yield 
\begin{eqnarray}\label{estim5_6}
\|\hat \be_y(\mathbf{v}^{ij}_\vartheta (t-s, s+ h))- \hat \be_y(\mathbf{v}^{ij}_\vartheta (t-s, s))\|^2_{L^2(0,T-h; L^2(\Omega_t\times \hat Y))}\leq Ch. \qquad
\end{eqnarray}
 Thus,   \eqref{estim_nb_Kolmog} and  \eqref{estim5_1}--\eqref{estim5_6} along with   the Kolmogorov theorem and the strong convergence and boundedness of $b^\vartheta$ ensure  
 \begin{equation*}\label{conver_macro_1}
 \begin{aligned}
 &\int_{\hat Y} \hat  \be_y(\mathbf{w}_\vartheta^{ij}) dy \to \int_{\hat Y} \hat \be_y(\mathbf{w}^{ij}) dy, \qquad \quad 
\int_{\hat Y_M} \hat  \be_y(\partial_t \mathbf{w}_\vartheta^{ij}) dy \to \int_{\hat Y_M}\hat \be_y(\partial_t \mathbf{w}^{ij}) dy \;   &&  \text{ in }  L^2(\Omega_T), \\
 &
 \widetilde{\mathbb E}^{\vartheta}_{\text{hom}}(b^\vartheta)  \to \widetilde{\mathbb E}_{\text{hom}}(b), \;\qquad \qquad \qquad \; \; \quad    {\mathbb E}^{\vartheta}_{\text{hom}}(b^\vartheta)  \to {\mathbb E}_{\text{hom}}(b)    && \text{ in }  L^2(\Omega_T),\\
 &\int_{\hat Y} \hat  \be_y(\boldsymbol{\chi}^{ij}_{\mathbb V, \vartheta}) dy \to \int_{\hat Y} \hat \be_y(\boldsymbol{\chi}^{ij}_{\mathbb V}) dy, 
 \quad \qquad  \mathbb V^{\vartheta}_{\text{hom}}(b^\vartheta)  \to   \mathbb V_{\text{hom}}(b)  \;   &&  \text{ in }  L^2(\Omega_T),\quad\\
 & \int_{\hat Y} \hat  \be_y(\mathbf{v}_\vartheta^{ij}(t-s,s)) dy \to \int_{\hat Y} \hat \be_y(\mathbf{v}^{ij}(t-s,s)) dy   && \text{ in }  L^2(0,T;L^2(\Omega_t)), \\ 
 &  \widetilde{\mathbb K}^{\vartheta}(t-s, s, b^\vartheta)  \to \widetilde{\mathbb K}(t-s,s, b) &&  \text{ in }  L^2(0, T; L^2(\Omega_t)),
 \end{aligned}
 \end{equation*}
as $\vartheta \to 0$,  where 
\begin{align*}
  \widetilde{\mathbb E}^{\vartheta}_{\text{hom}, ijkl}(b^\vartheta) &= \dashint_{\hat Y} \big[  \mathbb E_{ijkl}(b^\vartheta,y) + \big(\mathbb E(n_b^\vartheta,y) \hat  \be_y(\mathbf{w}_\vartheta^{ij})\big)_{kl} \big]dy, \\
 \widetilde{\mathbb K}^{\vartheta}_{ijkl}(t,s, b^\vartheta) &= \dashint_{\hat Y}   \big(\mathbb E(b^\vartheta(t+s),y) \hat \be_y(\mathbf{v}^{ij}_\vartheta(t,s))\big)_{kl} dy. 
\end{align*}
Using estimates \eqref{estim_w_chi} and \eqref{estim_v_theta}   and the strong convergence of $b^\vartheta$  yields 
  \begin{eqnarray*} 
  \int_0^{T-s} \mathbb K^{\vartheta}_{ijkl}(t,s,x,b^\vartheta(t,x)) dt  &\to & \int_0^{T-s} \mathbb K_{ijkl}(t,s,x, b(t,x))   dt
\end{eqnarray*}
as $\vartheta \to 0$, for a.a.\ $x \in \Omega_T$ and $s \in [0,T]$. Then, estimate \eqref{estim_hom_tensor} and  the Lebesgue dominated convergence theorem, implies 
$$ \int_0^{T-s} \mathbb K^\vartheta(t,s, b^\vartheta) dt \to \int_0^{T-s}  \mathbb K(t,s, b) dt  \quad \text{  in  } \; 
L^2(\Omega_T) \quad \text{ as } \; \vartheta \to 0. $$ 
 
 The strong convergence of  $\widetilde{\mathbb E}^{\vartheta}_{\text{hom}}$ and  $\widetilde{\mathbb K}^{\vartheta}$ and estimates \eqref{apriori_u_theta} ensure the strong convergence  
 $$\mathcal N_\delta^{\rm eff} ( \be(\bu^\vartheta)) \to \mathcal N_\delta^{\rm eff} ( \be(\bu)) \quad \text{ in } \; L^2(\Omega_T) \quad \text{ as } \;  \vartheta \to 0. $$

Hence, taking the limit as $\vartheta \to 0$ in the weak formulation of \eqref{macro_vis_nu}  we obtain the macroscopic equations \eqref{macro_vis}.
 Notice that for the integral-term in \eqref{macro_vis_nu} we have 
 \begin{eqnarray*}
\Big \langle   \int_0^t \mathbb K^\vartheta(t-s,s, b^\vartheta)  \partial_s \be(\bu^\vartheta (s,x))  ds, \boldsymbol{\psi}(t,x) \Big\rangle_{\Omega_T}  = 
  \int_{\Omega_T} \partial_s \be(\bu^\vartheta(s,x))  \int_0^{T-s} \mathbb K^\vartheta(\tau,s, b^\vartheta) \boldsymbol{\psi}(\tau+s,x)  d\tau dx ds
 \end{eqnarray*}
 for all $\boldsymbol{\psi} \in C^\infty(\Omega_T)^3$, $\boldsymbol{\psi}$ is $a_3$-periodic in  $x_3$. Thus, using the weak convergence of $\partial_s \be(\bu^\vartheta)$ and the strong convergence of $\int_0^{T-s} \mathbb K^\vartheta(t,s, b^\vartheta) d t$ we can pass to the limit in  the last term in  \eqref{macro_vis_nu}.

The assumptions on the elastic $\mathbb E(b, y)$ and viscoelastic $\mathbb V_M(b, y)$ tensors  
together with the regularity and boundedness  of $b$  ensure the existence of solutions of the unit cell problems \eqref{unit_chi} and  \eqref{unit_wv}. 
As before,   the assumptions on $\mathbb E$ and $\mathbb V_M$,  the boundedness of $b$, and the estimates \eqref{estim_w_chi}--\eqref{estim_v_theta}  yield   the symmetry properties and  strong ellipticity of $\mathbb V_{\text{hom}}$, see e.g.\ \cite{OShY}, as well as the boundedness of the macroscopic tensors, i.e.\ 
$\widetilde{\mathbb E}_\text{hom} \in L^\infty(0,T; L^\infty(\Omega))$, $\mathbb E_\text{hom} \in L^2(0,T; L^\infty(\Omega))$, 
$\mathbb V_\text{hom} \in L^\infty(0,T; L^\infty(\Omega))$,  $\widetilde{\mathbb K}(t-s,s)\in L^\infty(0,T; L^\infty(0,t; L^\infty(\Omega)))$,
and $\mathbb K(t-s,s)\in  L^2(0,T; L^\infty(0,t; L^\infty(\Omega)))$.
 This together with the assumptions on the coefficients and nonlinear functions in the equations for $\p, \n$, and $b$, see Assumetion~\ref{assumptions}, ensures the  existence of a unique weak solution of the macroscopic problem \eqref{macro_1}--\eqref{macro_vis}. 
Using estimates  \eqref{apriori_u_theta} we obtain $\bu \in H^1(0,T; \cW(\Omega))$. Hence,   $\bu \in C([0,T]; \cW(\Omega))$ and  $\bu$ satisfies the  initial condition $\bu(0,x) = \bu_0(x)$ for  $x \in \Omega$. 
%
\end{proof}

%
%
%

%
%

\end{document}